\documentclass[a4paper,11pt]{article}
\usepackage[utf8]{inputenc}
\usepackage{amssymb,amsmath,mathrsfs,amsthm}
\usepackage{graphics, graphicx}
\usepackage{geometry}
\usepackage{tikz}
\usepackage{pgfmath}
\usepackage{indentfirst} % for indent
\usepackage{verbatim}%for comment

\usepackage{hyperref} %navigation
\hypersetup{pdftex,colorlinks=true,allcolors=blue}
\usepackage{hypcap}
\usetikzlibrary{arrows}\definecolor{uuuuuu}{rgb}{0,0,0}\definecolor{ududff}{rgb}{0,0,0}\definecolor{xdxdff}{rgb}{0,0,0}

\newtheorem{thm}{Theorem}[section]
\newtheorem{defi}[thm]{Definition}
\newtheorem{prop}[thm]{Proposition}
\newtheorem{cor}[thm]{Corollary}

\newtheorem{lem}[thm]{Lemma}
\newtheorem{rem}[thm]{Remark}

\makeatletter
\newtheorem*{rep@theorem}{\rep@title}
\newcommand{\newreptheorem}[2]{%
	\newenvironment{rep#1}[1]{%
		\def\rep@title{#2 \ref{##1}}%
		\begin{rep@theorem}}%
		{\end{rep@theorem}}}
\makeatother

\newreptheorem{lem}{Lemma*}

\newcommand{\RomanNumeralCaps}[1]
{\MakeUppercase{\romannumeral #1}}

\numberwithin{equation}{section}
\numberwithin{figure}{section}

\renewcommand{\rm}[1]{\mathrm{#1}}
\renewcommand{\cal}[1]{\mathcal{#1}}

\newcommand{\bb}[1]{\mathbb{#1}}
\newcommand{\dd}{\mathrm{d}}

\newcommand{\supp}{\rm{supp}}

\newcommand{\delx}{-\delta(X)^2/4}
\newcommand{\hfn}{\bb{H}^n_{\bb{F}}}

\def\cusp{O}
\def\varbet{\exp^\perp}
\def\mani{M}
\def\l{\langle}
\def\r{\rangle}
\def\nexp{\exp^\perp }
\def\convex{D}
\def\eigen{\err(x)}
\def\err{c }
\def\normal{\mathbf{n}}
\def\opt{T}
\def\Rplus{\bb R_{\geq 0}}
\def\R {\bb R}
\def\two{\rm{\RomanNumeralCaps{2}}}
\def\M{Z}
\def\theone{\theta}
\def\thetwo{\iota}
\begin{document}

\bibliographystyle{alpha}
\title{\textbf{Finiteness of Small Eigenvalues of Geometrically Finite Rank one Locally Symmetric Manifolds}}
\author{LI Jialun}
\date{}
\maketitle

\begin{abstract}
	Let $M$ be a geometrically finite rank one locally symmetric manifolds. We prove that the spectrum of the Laplace operator on $M$ is finite in a small interval which is optimal. 
\end{abstract}

\section{Introduction}
\label{sec:intro}

Let $M$ be a complete Riemannian manifold. The Laplace operator $\Delta$ acts on the compactly supported smooth functions and admits an extension to an unbounded self-adjoint operator on $L^2(M)$. The study of the spectrum of the Laplace operator on geometrically finite hyperbolic manifolds was started by Lax and Phillips in \cite{laxphillips82}. The motivation is to give an exponential error term in the asymptotic distribution of orbits for discrete subgroups of the group of isometries of $\tilde{M}$. They proved that on a geometrically finite real hyperbolic manifold of dimension $n$, the intersection of the interval $(-(n-1)^2/4,0]$ and the spectrum consists of at most finitely many eigenvalues. Geometrical finiteness means that the quotient manifold has a fundamental domain which is a finitely sided polyhedron. 

For general cases (see \cite{hamenstadt2004}), Hamenst\"adt proved that on geometrically finite rank one locally symmetric manifolds $\Gamma\backslash X$, where $X=\hfn$ for $\bb F=\bb R,\bb C,\bb H $ or $\hfn=\bb{H}^2_{\bb O}$, the intersection of the interval $[\delx+\chi,0]$ and the spectrum also consists of at most finitely many eigenvalues, where $\chi$ is any positive number and the volume entropy $\delta(X)$ is $(n+1)\dim_{\bb R}\bb F-2$. In this case, the applications to counting problems were given in \cite{kim2015counting}. Inspired by the method in \cite{laxphillips82}, we generalize the result of Lax and Phillips to the case of rank one locally symmetric manifolds.
\begin{thm}\label{the:main}
Let $M=\Gamma\backslash X$ be a geometrically finite rank one locally symmetric manifold. Then the intersection of the spectrum of the Laplace operator and the critical interval $(\delx,0]$ consists of finitely many eigenvalues of finite multiplicities.
\end{thm} 
\begin{rem}
1. If $M$ has infinite volume, the result is optimal in the following sense. Under the same assumption as in Theorem \ref{the:main}, the interval $(-\infty,\delx]$ is contained in the essential spectrum of $\Delta$. This can be proved by using \cite[Theorem 4.4]{castillon2014harmonic} (or \cite{donnelly1981essential} and \cite[Prop 7.2]{borthwick2007spectral}). Together with our result, this implies that the essential spectrum of $\Delta$ is exactly $(-\infty,\delx]$ when the volume is infinite.

2. In the convex cocompact complex hyperbolic case, the meromorphic extension of the resolvent is already known from \cite[Appendix A7]{epstein1991resolvent}. The finiteness of the spectrum in the critical interval is a consequence of the meromorphic extension, as in \cite[Prop.7.3]{borthwick2007spectral}. The main idea is that an eigenvalue of the Laplace operator in the critical interval corresponds to a pole of the resolvent in $[0,\delta(X)/2)$.

3. In the general convex cocompact rank one case, the finiteness of the discrete spectrum is due to \cite{bunkeolbrich_2000_kleinian}. Later, they published a paper \cite{bunkeolbrich_2008_geometric} to treat the geometrically finite rank one case. But they made an additional assumption for quaternion case and they didn't treat the Cayley hyperbolic case. Our result is new in these two cases.
\end{rem}
Unless otherwise stated we assume that the manifolds, the Riemannian metrics and the functions are $C^\infty$ smooth. 
\section*{Acknowledgements}
The author wishes to express his gratitude to Jean-François Quint for many stimulating conversations and for his precious comments on a previous version of this
manuscript. The author is grateful to the referee for carefully reading the manuscript and for helpful remarks.
\section{Estimates for the spectrum on Riemannian manifolds}
\subsection{Barta's trick}
Here we will use Barta's trick, a way to estimate the bottom of the spectrum of the Laplace operator by one good function (see \cite[Lem.3.3]{roblin2015critical}). Recall that for every complete Riemannian manifold $\mani$, the bottom of the spectrum of the Laplace operator $\lambda_0(\mani)$ equals the infimum of the Rayleigh quotients $\cal R(f)=\frac{\int\|\nabla f\|^2\dd vol}{\int f^2\dd vol}$, over all nonzero smooth functions on $\mani$ with compact support. Hence it is natural to consider the quantity $\int (\|\nabla f\|^2-\lambda f^2)\dd vol$.
\begin{lem}\label{lem:barta}
	Let $u, \varphi$ be two real smooth functions on a Riemannian manifold $\mani$, the support of $u$ being compact. We have
	\begin{equation}\label{equ:barta}
	\int \|\nabla (\varphi u)\|^2\dd vol=\int \varphi^2\|\nabla u\|^2\dd vol-\int u^2\varphi\Delta\varphi\dd vol .
	\end{equation}
\end{lem}
Please see \cite[Lemma 3.3]{roblin2015critical}
\begin{comment}
	By definition, we have
	\begin{align*}
	\|\nabla(\varphi u)\|^2=\varphi^2\|\nabla u\|^2+u^2\|\nabla \varphi\|^2+\frac{1}{2}\l\nabla\varphi^2,\nabla u^2\r.
	\end{align*}
	By Stokes' formula, we have 
	\begin{equation*}
	\int \frac{1}{2}\l\nabla\varphi^2,\nabla u^2\r=\frac{1}{2}\int -u^2\Delta\varphi^2.
	\end{equation*}
	We also have $\Delta\varphi^2=2\varphi\Delta\varphi+2\|\nabla\varphi\|^2$. Gathering the three equalities above, we obtain the formula.
\end{comment}
%
\begin{prop}\label{prop:barta}
	Let $\mani$ be a Riemannian manifold. Let $f$ be a real smooth function on $\mani$ with compact support. Assume that for some $\lambda>0$ there exists a smooth function $\varphi$ such that $\varphi >0$ and $-\Delta\varphi\geq \lambda\varphi$. Then
	\begin{equation*}
	\int \|\nabla f\|^2\dd vol-\lambda\int f^2\dd vol\geq \int \varphi^2\|\nabla(f/\varphi)\|^2\dd vol.
	\end{equation*}
\end{prop}
\begin{proof}
	Applying Lemma \ref{lem:barta} with $u=f/\varphi$ implies
	\begin{align*}
	\int \|\nabla f\|^2\dd vol&=\int \varphi^2\|\nabla(f/\varphi)\|^2\dd vol-\int (f/\varphi)^2\varphi\Delta \varphi\dd vol\\
	&\geq \int \varphi^2\|\nabla(f/\varphi)\|^2\dd vol+\lambda\int f^2\dd vol.
	\end{align*}
	The proof is complete.
\end{proof}
\begin{rem}
	Compared with \cite[Proposition 3.2]{roblin2015critical}, we keep the last term $\int \varphi^2\|\nabla(f/\varphi)\|^2$. This term is important and will be exploited in Section \ref{sec:keynew}.
\end{rem}
The following proposition says that for a Riemannian manifold with quasi-warped product structure, the derivative of the volume density along the vertical geodesic can give us a good control of the bottom of the spectrum.
\begin{prop}\label{prop:product}
	Let $\mani=L\times \bb R$ be a Riemannian manifold with the metric given by 
	\begin{equation}\label{equ:gddt}
	g=\dd t^2+g(x,t),
	\end{equation}
	where $g(x,t)$ is a metric on $L\times\{ t\}$. We call this a quasi-warped product metric.
	
	The Riemannian volume element can be written as $\omega=\dd t\times h(x,t)\dd x$, where $\dd x$ is the volume element on $L\times \{ 0\}$ and $h$ is a density function from $L\times \bb R$ to $\bb R_{\geq 0}$ . Assume that for some $\lambda>0$ and a nonnegative function $\err$ on $\bb R$, we have 
	\begin{equation}\label{equ:parth}
	\partial_t h\geq 2\lambda(1-\err(t)) h.
	\end{equation}
	Then for $f\in C_c^\infty(\mani)$, we have
	\begin{equation}\label{equ:product}
	\int \|\nabla f\|^2\dd vol-\lambda^2\int(1-2\err(t)) f^2\dd vol\geq \int e^{-2\lambda t}\|\nabla(fe^{\lambda t})\|^2\dd vol.
	\end{equation}
\end{prop}
\begin{rem}
	The terminology quasi-warped product is a generalization of warped product in Riemannian geometry. Let $(\mani_1,g_1),(\mani_2,g_2)$ be two Riemannian manifolds. The warped product $\mani_1\times_f\mani_2$ is the product manifold $\mani_1\times \mani_2$ equipped with the warped product metric given by 
	\[g=g_1+fg_2, \text{ where }f \text{ is a positive function on }\mani_1. \]
\end{rem}
\begin{proof}
	Let $\varphi(x,t)=e^{-\lambda t}$. We will prove that $-\Delta\varphi\geq \lambda^2(1-2\err(t))\varphi$ on $L\times \bb R$. It is sufficient to prove this inequality locally.
	
	Take a local chart on $L$ by $\phi:\bb R^n\supset U\rightarrow L$. Then $\tilde{\phi}=(\phi, Id):U\times \bb R\rightarrow L\times \bb R$ is a local chart on $\mani$. We write elements in $U\times \bb R$ by $(x_1,\dots, x_n, x_{n+1})$ with $x_{n+1}=t$. By the definition of $\Delta$, we have
	\begin{align*}
	\Delta\varphi=\frac{1}{\sqrt{|\det(g)|}}\sum_{1\leq i,j\leq n+1}\partial_i(\sqrt{|\det(g)|}g^{ij}\partial_j \varphi),
	\end{align*}
	where $g^{ij}$ is the inverse matrix of the matrix of metric. By definition, we have $$h(x,t)=\sqrt{|\det (g_{x,t})|}.$$
	%Hence $h(x,t)$ is the product of $\sqrt{|\det(g)|}$ with a function independent of $t$.
	Since $\varphi$ is a function only depending on $t$ and the formula of metric implies $g^{tj}=0$ when $j\neq t$, we have
	\begin{align}\label{equ:deltavarphi}
	\Delta\varphi=\frac{1}{\sqrt{|\det(g)|}}\partial_t(\sqrt{|\det(g)|}\partial_t \varphi)=\frac{1}{h}\partial_t(h\partial_t \varphi)=\partial_{tt}\varphi+\frac{\partial_th\partial_t\varphi}{h}.
	\end{align}
	Due to $\partial_th\geq 2\lambda(1-\err(t)) h$ and $\partial_t\varphi=-\lambda\varphi<0$, we have
	\begin{align*}
	-\Delta\varphi\geq -\partial_{tt}\varphi-2\lambda(1-\err(t)) \partial_t\varphi=\lambda^2(1-2\err(t))\varphi.
	\end{align*}
	Applying Lemma \ref{lem:barta} with this $\varphi$ and $u=f/\varphi$, we obtain the result.
\end{proof}
\subsection{The Lax-Phillips inequality }
\label{sec:keynew}
Compared with \cite{hamenstadt2004}, where she fully used the information from the spectrum of the Laplace operator on each component and the negative curvature, the key new input in our article is the observation that the particular form of the Laplace operator gives us more information. A version on $\bb R$ is as follows:
\begin{prop}\label{prop:lapr}
	Let $\Delta=\partial_{tt}$ be the standard Laplace operator on $\bb R$. For every $C>0$, the positive spectrum of $\Delta+Ce^{-|t|}$ on $L^2(\bb R)$ is finite.
\end{prop}
This is a classical result in spectral theory (see \cite[Thm 8.5.1]{davies1996spectral} for a similar version, which says that the positive spectrum is at most countable). In Davies' book, the result comes from a compact perturbation. Another way to prove this type of result is to use the following observation, because the exponential function $e^{-|t|}$ has a very rapid decay. We will give a proof of Proposition \ref{prop:lapr} in Section \ref{sec:finite}.
\begin{lem}\label{lem:lpi}
	For every $C_0>0$, there exist a compact interval $U$ and $C_1>0$ such that for every real smooth, compactly supported function $f$ on $\bb R$ we have
	\begin{align}\label{ineq:lpi}
	\int \|\nabla f\|^2+C_1\int_Uf^2\geq C_0\int e^{-|t|}f(t)^2\dd t.
	\end{align}
\end{lem}
	We call \eqref{ineq:lpi} the Lax-Phillips inequality (LPI). This is a consequence of the following elementary lemma.
\begin{lem}\label{lem:derivative}
		If $g$ is a real smooth bounded function on $\bb R$, then for all $r$ in $\bb R_{\geq 0}$ we have
		\begin{equation}\label{ineq:derivative}
		\int_{r}^{\infty}g(t)^2e^{-t}\dd t\leq 2\left(  g(r)^2e^{-r}+\int_{r}^{\infty}g'(t)^2(1+t-r)e^{-t}\dd t\right) .
		\end{equation}
\end{lem}
\begin{proof}
By the Newton-Leibniz formula,
\begin{equation}\label{equ:gtr}
g(t)=g(r)+\int_{r}^{t}g'(s)\dd s.
\end{equation}
By the Cauchy-Schwarz inequality,
\begin{equation}\label{equ:gt2}
g(t)^2\leq 2g(r)^2+2\left(\int_{r}^{t}g'(s)\dd s\right)^2\leq 2g(r)^2+2(t-r)\int_{r}^{t}g'(s)^2\dd s.
\end{equation}
Combining \eqref{equ:gtr} and \eqref{equ:gt2}, we have
\begin{align*}
\int_{r}^{\infty}g(t)^2e^{-t}\dd t&\leq \int_{r}^{\infty}\left(2g(r)^2+2(t-r)\int_{r}^{t}g'(s)^2\dd s\right) e^{-t}\dd t\\
&=2\left( g(r)^2e^{-r}+\int_{r}^{\infty}g'(s)^2(1+s-r)e^{-s}\dd s\right) .
\end{align*}
The proof is complete.
\end{proof}
\begin{proof}[Proof of LPI (Lemma \ref{lem:lpi})]
	We divide $\bb R$ into $\bb R_{\geq 0}$ and $\bb R_{<0}$. By symmetry, we only need to prove the inequality on $\bb R_{\geq 0}$. Let $T$ be a large number such that $e^{T}\geq C_0$. Then by monotonicity, we have $(1+t-T)e^{-t}\leq 1/C_0$ for all $t\geq T$. Integrating \eqref{ineq:derivative} on $[T,T+1]$ with $g=f$, we have
	\[\int_{T+1}^{\infty}f(t)^2e^{-t}\dd t\leq \int_{T}^{T+1}f(t)^2e^{-t}\dd t+\frac{1}{C_0}\int_{T}^{\infty}f'(t)^2\dd t.\]
	The result follows by taking $U=[0,T+1]$ and $C_1=2C_0$.
\end{proof}
This idea is already utilized in Theorem 3.3 of \cite{laxphillips82}. The error term $\int e^{-|x|}f^2$ is not artificial, which will appear naturally when we estimate the spectrum on the complement of the convex core.
\begin{cor}\label{cor:lpi}
	Under the same assumption as in Proposition \ref{prop:product}, if there exists $C_0>0$ such that $e^{2\lambda t}/C_0\leq h(t)\leq C_0e^{2\lambda t}$ for $t\in\bb R$, then for every $C>0$, there exist a compact interval $I$ and a positive constant $C_1$ such that the following holds. For every smooth, compactly supported function $f$ on $L\times \bb R$ and every compact subset $K$ of $L$, we have
	\begin{equation*}
		\int (\|\nabla f\|^2-\lambda^2 f^2)\dd vol+C_1\int_{K\times I}f^2\dd vol\geq C\int_{K\times\bb R} e^{-|t|}f^2\dd vol.
	\end{equation*} 
\end{cor}
\begin{proof}
	Let $f_1=fe^{\lambda t}$.
	Due to the quasi-warped product structure of the Riemannian metric,
	\begin{equation}\label{equ:gra1}
	\|\nabla f_1\|^2=\|\nabla'f_1\|^2+\partial_t f_1^2\geq \partial_t f_1^2,
	\end{equation}
	where $\nabla'$ is the gradient on $L\times \{t \}$.
	Substituting \eqref{equ:gra1} into LPI (Lemma \ref{lem:lpi}), with the constant $C_2=C_0^2(C+2\lambda^2)$, implies that there exist $I$ compact, $C_1>0$ such that
	\begin{equation}\label{ineq:lpi1}
	\int_{\bb R} \|\nabla f_1(x,t)\|^2\dd t+C_1\int_If_1(x,t)^2\dd t\geq C_2\int_{\bb R} e^{-|t|}f_1^2(x,t)\dd t, \text{ everywhere on }L.
	\end{equation}
	Integrating \eqref{ineq:lpi1} over $K$ with respect to $\dd x$, we get
		\[\int_{K\times \bb R} \|\nabla f_1\|^2 \dd x\dd t+
		C_1\int_{K\times I}f_1^2\dd x\dd t\geq (C+2\lambda^2)\int_{K\times \bb R} e^{-|t|}f_1^2\dd x\dd t.\]
	Using $\dd vol=h(x,t)\dd t\dd x$ and $e^{2\lambda t}/C_0\leq h(t)\leq C_0e^{2\lambda t}$, we obtain
	\[\int_{K\times \bb R} \|\nabla f_1\|^2 e^{-2\lambda t}\dd vol+
	C_1\int_{K\times I}f^2\dd vol\geq (C+2\lambda^2)\int_{K\times \bb R} e^{-|t|}f^2\dd vol .\]
	This formula together with Proposition \ref{prop:product} implies the result.
\end{proof}
\section{Finiteness of the spectrum}
\label{sec:finite}
\begin{prop}\label{prop:main}
	Let $M$ be a complete Riemannian manifold. If for some smooth bounded function $\eigen\geq0$, there exists a compact subset $U$ with smooth boundary and $\epsilon, C_U>0$ such that the following holds. For any compact subset $V$ there is $\epsilon_V>0$ such that for all complex valued function $f\in C_c^\infty(M)$ we have
	\begin{equation}\label{ineq:main}
	\int (\|\nabla f\|^2-\eigen|f|^2)\dd vol+C_U\int_{U}|f|^2\dd vol\geq \epsilon\int_U \|\nabla f\|^2\dd vol+\epsilon_V\int_V|f|^2\dd vol.
	\end{equation}
	Then the positive spectrum of the operator $\opt =\Delta+\eigen$ consists of at most finitely many eigenvalues of finite multiplicities.
\end{prop}
\begin{proof}[Proof of Proposition \ref{prop:lapr} from Proposition \ref{prop:main}]
	From Lemma \ref{lem:lpi}, we have
	\begin{equation*}
		\int \left(\|\nabla f\|^2-\frac{C_0}{4}e^{-|t|}f^2\right)\dd t+\frac{1}{2}C_1\int_Uf^2\geq \frac{1}{2}\int \|\nabla f\|^2+\frac{C_0}{4}\int e^{-|t|}f(t)^2\dd t.
	\end{equation*}
	Take $C_0=4C$ and $c(x)=Ce^{-|x|}$. Then $c(x)$ satisfies \eqref{ineq:main}. Therefore,
	Proposition \ref{prop:lapr} follows from Proposition \ref{prop:main}.
\end{proof}
This is a ``baby case" of the main result of this manuscript, whose proof will also follow from Proposition \ref{prop:main}. We will establish \eqref{ineq:main} for geometrically finite rank one locally symmetric manifolds in Section \ref{sec:geofinman}.
 
It remains to prove Proposition \ref{prop:main}. Recall some results in spectral theory:
\begin{defi}
	Let $H$ be a complex Hilbert space, $\opt $ be a linear operator with the domain of definition $D$, which is dense in $H$. We call $\opt $ self-adjoint if the domain of definition of its adjoint $\opt ^*$ satisfies $D^*=D$ and $\opt ^*$ equals $\opt $ on $D$, where $D^*$ is defined as
	\[D^*=\{f\in H|\exists C_f>0 \text{ such that }\forall g\in D,\ |( f,\opt g)|\leq C_f|g|_H \}. \]
	
	We call $T$ positive if for every nonzero $f$ in $D$, we have $(Tf,f)> 0$.
\end{defi}
\begin{prop}
	Let $M$ be a complete Riemannian manifold. Define the space $H^1_0(M)$ as the completion of $C^{\infty}_c(M)$ under the norm
	\begin{equation*}
	\| f\|_{H^1}=\| f\|_{L^2}+\|\nabla f\|_{L^2}.
	\end{equation*}
	Let the domain of the Laplace operator be
	\begin{equation*}
	D=\{ f\in H^1_0(M)|\ \Delta f\in L^2(M)\}.
	\end{equation*}
	Then $\Delta:D\subset L^2\rightarrow L^2$ is a self-adjoint operator.
\end{prop}
(See \cite[Sec.8.2]{taylor2} for more details.) 
For $f\in C_c^\infty(M)$, define
\begin{align*}
K(f)=C_U\int_U|f|^2\dd vol,\ E(f)=\int(\|\nabla f\|^2-\eigen|f|^2)\dd vol, 
\end{align*} 
and $F(f)=E(f)+K(f)$.
Inequality \eqref{ineq:main} implies that $F$ is a positive definite quadratic form on $C^{\infty}_c(M)$. We define a Hilbert space $\cal H$ as the completion of $C_c^\infty(M)$ with respect to the norm $|\cdot |_F$, written as $\mathcal{H}=\overline{C^{\infty}_c(M)_F}$. 

For an open subset $V$ of $M$, we define $H^1(V)$ as the completion of $C^\infty(V)$ with respect to the norm $\|f\|_{H^1( V)}^2=\int_V|f|^2+\int_V\|\nabla f\|^2$.
\begin{prop}\label{prop:finicodimension}
	With the same assumption as in Proposition \ref{prop:main}, there exists a subspace $\cal H_1$ of finite codimension in $\cal H$,
	on which $E$ is positive definite.
\end{prop}
\begin{proof}
	Because of inequality \eqref{ineq:main}, we can define a bounded restriction map from $\cal H$ to $H^1(U)$ and compose it with the injection $\iota$ from $H^1(U)$ to $L^2(U)$, that is
	\begin{equation*}
	S:\cal{H}\rightarrow H^1(U)\stackrel{\iota}{\longrightarrow}L^2(U).
	\end{equation*}
	Let $S^*$ be the adjoint of $S$, then $S^*S:\cal{H}\rightarrow\cal{H}$ is a self-adjoint operator. For $f\in \cal H$, we have
	$$F(S^*Sf,f)=\int_U|f|^2\ \dd vol=\frac{1}{C_U}K(f).$$ 
	By the Rellich theorem \cite[Chapter 4, Proposition 4.4]{taylor1} and the smoothness of the boundary of $U$, the injection $\iota$ is compact. Therefore $S^*S$ is a compact self-adjoint operator. The set of eigenvalues has a unique accumulation point $0$. In a subspace of finite codimension in $\cal{H}$, we have $|K(f)|=C_U|F(S^*Sf,f)|\leq \frac{1}{2}|F(f)|$. Therefore
	$$E(f)=F(f)-K(f)\geq\frac{1}{2}F(f).$$
	The proof is complete.
\end{proof}
\begin{prop}\label{prop:selfadjoint}
	If $T$ is a self-adjoint operator on a Hilbert space $H$, then there is a decomposition
	\[H=H_+\oplus H_-\oplus \ker T \]
	such that $T$ preserves the decomposition, $T$ is self-adjoint on $H_+$, $H_-$ and $T$ is positive, negative on $H_+,H_-$, respectively.
	
	If there exists $\lambda>0$ such that $\lambda Id-T$ is positive, then $H_+$ in the above decomposition is actually in $D$, the domain of definition.
\end{prop}
(See \cite[32 Thm1]{lax1} for more details.)
\begin{proof}[Proof of Proposition \ref{prop:main}]
	By definition,
	$$\|f\|_F\leq\|f\|_{L^2}+\|\nabla f\|_{L^2}=\|f\|_{H^1}.$$
	Thus we can extend the injection $C^{\infty}_c(M)\rightarrow\cal{H}$ to an application $j:H^1_0(M)\rightarrow\cal{H}$. By \eqref{ineq:main}, for any compact subset $V$ of $M$, we have $F(f)\geq \epsilon_V \int_V |f|^2$ for $f$ in $C_c^\infty(M)$. Therefore $j$ is injective and can be seen as an inclusion.
	
	Since $\Delta$ is self-adjoint and $c$ is bounded, the operator $\opt =\Delta+\eigen$ is also self-adjoint. Since $(\|c\|_\infty +1)Id-T$ is positive, by using Proposition \ref{prop:selfadjoint} with $\opt $ and $H=L^2(M)$, we have $H_+\subset D\subset H^1_0(M)\subset \cal H$. For a nonzero element $u$ in $H_+$, we have 
	\begin{align*}
	E(u)&=\int_M (\|\nabla u\|^2-\eigen|u|^2)\dd vol=\int_M-(\opt u)\bar u\dd vol=-(Tu,u)<0.
	\end{align*}
	Proposition \ref{prop:finicodimension} implies $\cal H_1\cap H_+=\{0\}$, therefore $H_+$ is of finite dimension. Due to Proposition \ref{prop:selfadjoint}, the operator $\opt $ is self-adjoint on $H_+$, hence the positive spectrum of $\opt $ is finite and each element is an eigenvalue of finite multiplicity. 
\end{proof} 
\section{Rank one locally symmetric manifolds}
We study the spectrum on cusps and complement of convex sets in this section, which will be used in Section \ref{sec:geofinman} to obtain global result.
\subsection{Real rank one globally symmetric spaces}
\label{sec:rankone}
Real rank one globally symmetric spaces are usually classified into four families: $\bb H^n_{\bb R}=SO_o(1,n)/SO(n) $, $ \bb H^n_{\bb C}=PU(1,n)/U(n)$, $\bb H^n_{\bb H}=PSp(1,n)/(Sp(1)\times Sp(n))$ and an exceptional one, the Cayley hyperbolic plane $\bb H^2_{\bb O}=F_4/Spin(9)$. For the first three types, we have a uniform treatment by projective models, in which the metric and the curvature can be computed explicitly (see for example \cite{mostow1973strong}, \cite{pansu1989carnot}, \cite{quintpatterson}). But the Cayley hyperbolic plane needs different model. 

The space $\hfn$ is a Riemannian manifold with pinched curvature between $-1$ and $-4$. 
The ideal boundary of the symmetric space $\hfn$, denoted by $\partial\hfn$, is the set of equivalent classes of geodesic rays. From now on, abbreviate $\hfn,\ \partial \hfn,\ \hfn\cup\partial\hfn$ to $X,\  X_I,\ X_c$. Equip $X_c$ with the topology such that $X_c$ becomes a compact manifold with boundary. The space $X$ is homeomorphic to an open ball of dimension $n\dim_{\bb R}\bb F$, and $X_c$ is homeomorphic to a closed ball of the same dimension. 
Let $\xi$ be a point in $X_I$, and let $x, y$ be two points in $X$. We define the Busemann function by
$$b_{\xi}(x,y)=\lim_{t\rightarrow+\infty}(d(x,\gamma(t))-d(y,\gamma(t))),$$
where $d$ is the distance induced by the Riemannian metric, and $\gamma(t)$ is a geodesic ray asymptotic to $\xi$. Let $o$ be a fixed reference point in $X$. The level sets of $b_\xi(\cdot,o)$ are called horospheres based at $\xi$. A horoball based at $\xi$ is a set $b_\xi(\cdot,o)^{-1}(-\infty,r]$.

We introduce the horospherical model for $ X$ following \cite[Section 9]{pansu1989carnot}. Fix a point $\infty$ of $X_I$. Let $G$ be the group of isometries of the symmetric space $X$.  Let $G=KAN$ be the Iwasawa decomposition with $A$ and $N$ fixing $\infty$, and $K$, a maximal compact subgroup, fixing $o$. The group $A$ is isomorphic to $\bb R$ and $N$ is a simply connected nilpotent Lie group. Let $\M=K\cap Stab_\infty(G)$. Here $\M$ is also the maximal subgroup of $K$ which commutes with $A$. Let $\cal N$ be the Lie algebra of $N$. The group $A$ normalizes $N$, and the conjugation action of $A$ induces an automorphism on $\cal N$. We make a choice of a particular generator of $A$ such that $\cal N$ admits a decomposition 
\[\cal N=V^1\oplus V^2, \]
where $V^j$ equals to $Ker(\rm{Ad}(a_t)-e^{jt})$ for $t\neq 0$ and $a_t$ is an element in $A$. Moreover, $V^2$ is the centre of $\cal N$ and $[V^1,V^1]=V^2$.
We can identify $V^1, V^2$ with $\bb F^{n-1}, \Im\bb F$, and the Lie algebra structure is given by
\[[(x_1,\dots,x_{n-1}),(y_1,\dots, y_{n-1})]=\Im \sum_{j=1}^{j=n-1}\bar{x_j}y_j. \] 
Starting from $V^1\oplus V^2$, we obtain two left invariant distributions $W^1,W^2$ on $N$.
Let $\alpha$ be the endomorphism on $\cal N$, which is the differential of $\rm{Ad}(a_t)$ at $t=0$, mapping vector $v$ in $V^j$ to $jv$. Let $e^\alpha$ be the induced automorphism on $N$. Let $c(t)$ be the geodesic ray starting from $o$ to $\infty$, that is $c(t)=a_to$. We have a diffeomorphism from $N\times A$ to $X$ given by $(\nu,t)\rightarrow \nu\cdot c(t)$, and the action of $A$ reads as $a_t(\nu,s)=(e^{t\alpha}\nu,t+s)$. Let $m=(n-1)\dim_{\bb R}\bb F$ and $q=\dim_{\bb R}\bb F-1$.
The hyperbolic metric on $X$ can be written as
\begin{equation}\label{equ:duamet}
	g=g_t\oplus \dd t^2,
\end{equation}
where these left invariant metrics $g_t$ on $N$ have, under the distribution $W^1,W^2$, a matrix of the form
\begin{equation*}
\begin{pmatrix}
	e^{2t}Id_{m} & 0 \\ 0 & e^{4t}Id_{q}
\end{pmatrix}.
 \end{equation*}
 
We also need to calculate the Riemann curvature tensor. Let $\normal$ be the unit tangent vector at $o$ to the geodesic $c(t)$. If $v\in V^1$, then $\normal,v$ are tangent to a totally real $2-$plane of curvature $-1$. If $v\in V^2$, then $\normal,v$ are tangent to a $\bb F-$line of curvature $-4$. Therefore
\begin{equation}\label{equ:riecur}
R(\normal,v_j,\normal)=-j^2v_j\text{ for }v_j\in V^j\ \ j=1,2.
\end{equation}

We introduce a half space model. As in the horospherical model, with $o$ in $X$ and $\infty$ in $X_I$ fixed, let $y=e^{-t}$. The coordinate map is replaced by a map from $N\times \bb R_{>0}$ to $X$, that is $(\nu,y)\rightarrow \nu\cdot c(\log y)$. 
We also give a differential structure to $X_c$. Locally the compactification is obtained by adding the hyperplane $\{ y=0 \}$. Locally, this also gives a differential structure to the manifold with boundary. If we fix another point in $X_I$, and give a differential structure by the same procedure, then by the compatibility of the differential structures, we have a differential structure on $X_c$.
\subsection{Discrete subgroups}
\label{sec:geofin}
Recall that $G$ is the group of isometries of the symmetric space $X$, and let $\Gamma$ be a discrete subgroup of $G$. The group $\Gamma$ may have torsion. In the geometrically finite case, the group $\Gamma$ is always finitely generated \cite[Prop5.5.1]{bowditch1995}. We can use a result of Selberg \cite{selberg60discontinuous}, passing to a normal subgroup $\Gamma'$ which is of finite index in $\Gamma$ and has no torsion. Since the spectrum of the Laplace operator on a finite covering space contains the spectrum of the original one, we suppose that $\Gamma$ is without torsion. 

 Fix a point $x$ in $X$, and let $\overline{\Gamma\cdot x}$ be the closure of the orbit $\Gamma\cdot x$ in $X_c$. The limit set of the group $\Gamma$ is defined by
$$\Lambda(\Gamma)=X_I\cap\overline{\Gamma\cdot x}.$$
This definition of the limit set is independent of the choice of $x$. Let $\Omega=X_I-\Lambda(\Gamma)$. Since the actions of $\Gamma$ on $X\cup\Omega$ and $X$ are properly discontinuous \cite[Prop3.2.6]{bowditch1995}, we set
$$M_c(\Gamma)=\Gamma\backslash(X\cup\Omega),\ M(\Gamma)=\Gamma\backslash X.$$ 
These are a manifold with boundary and a rank one locally symmetric manifold, respectively. 

In order to study the spectrum, we define an energy form, which will be used throughout the paper. Let
\begin{equation}\label{equ:energy}
E(f)=\int_{M(\Gamma)}(\|\nabla f\|^2-\ell^2f^2)\dd vol
\end{equation} 
for $f\in C_c^\infty(M(\Gamma))$, where $\Gamma$ is a discrete subgroup of $G$ and $\ell=\delta(X)/2=((n+1)\dim_{\bb R}\bb F-2)/2$ is half the volume entropy.
\subsection{Cusps}
\label{sec:cusreg}
\begin{defi}[Parabolic subgroup]
	Let $\Pi$ be a subgroup of $G$. We call $\Pi$ parabolic if the set of fixed points of $\Pi$ in $X_c$ consists of a unique point $\xi$ in $X_I$, and $\Pi$ preserves setwise every horosphere based at $\xi$.
\end{defi}
Let $\Pi$ be a discrete parabolic subgroup with fixed point $\infty$, a point in $X_I$. We use the horospherical model $ N\times A\rightarrow X$ introduced in Section \ref{sec:rankone}. Then $\Pi $ is a subgroup of $\M\ltimes N$. Recall that $\M$ is the subgroup of $K$ which preserves setwise every horosphere based at $\infty$ and the metric $g_t$ on $N$. The part $N$ of $\M\ltimes N$ acts as translation on $N$ and $\M$ acts as rotation. Let $\Pi\backslash N$ be the quotient space of $N$ under the action of $\Pi$. The quotient manifold satisfies
\[M(\Pi)\simeq(\Pi\backslash N)\times \bb R\text{ as a topologic space.} \]

We will apply Proposition \ref{prop:product} to $M(\Pi)$. Since $\Pi$ preserves the metric $g_t$, we have a quotient metric on $\Pi\backslash N$. The formula \eqref{equ:duamet} implies that the metric on $(\Pi \backslash N)\times \bb R$ is a quasi-warped product metric, and the volume element is equal to $e^{2\ell t}\dd t\dd\eta$, where $\dd\eta$ is the volume element on $(\Pi \backslash N)\times \{0 \}$. So the function $h$ in Proposition \ref{prop:product} equals $e^{2\ell t}$, and $h$ satisfies \eqref{equ:parth} with $\lambda=\ell$ and $c(t)=0$. Applying Proposition \ref{prop:product} and Corollary \ref{cor:lpi}, we obtain 
\begin{lem}\label{prop:cusp}
	Let $E$ be as in \eqref{equ:energy} with $\Gamma=\Pi$. For a function $f\in C_c^{\infty}((\Pi \backslash N)\times \bb R)$, we have 
	\begin{equation}\label{ineq:cusp0}
	E(f)=\int \|\nabla (e^{\ell t}f)\|^2e^{-2\ell t}\dd vol.
	\end{equation}
	Given $C>0$, there exist a compact interval $I\subset\bb R$ and a constant $C_1>0$ such that for all compact set $K$ in $\Pi \backslash N$ we have
	\begin{equation}\label{ineq:cusp1} 
	\begin{split}
	E(f)+C_1\int_{K\times I}f^2\dd vol\geq C\int_{K\times \bb R_{\geq 0}} e^{-t}f^2\dd vol.
	\end{split}
	\end{equation}
\end{lem}
\subsection{Convex subsets and the normal exponential map} 
\label{sec:comcov}
We need a lemma which says that in the normal exponential coordinate, the Riemannian manifold has a quasi-warped product metric. This lemma is similar to the Gauss lemma, where the hypersurface $S$ degenerates to a point.
\begin{lem}\label{lem:normal}
	Let $S$ be a smooth hypersurface of a Riemannian manifold $M$. Let $\nexp :S\times \bb R_{\geq 0}\rightarrow M$ be the normal exponential map given by $\nexp (x,t)=\exp_x(t\normal(x))$. Assume $\nexp$ is an embedding. Then for every $x$ in $S$ and $s\geq 0$, the curve $\gamma_x:t\rightarrow \exp_x(t\normal(x))$ is normal to the hypersurfaces $\nexp (S\times \{s \})$. 
\end{lem}
\begin{defi}
	Let $M$ be a complete Riemannian manifold, and let $D$ be a closed subset of $M$. We call $D$ geodesically convex if the preimage $\tilde{D}$ of $D$ in the universal cover $\tilde{M}$ is convex, that is for any two points $x,y$ in $\tilde D$ all geodesics between $x$ and $y$ are contained in $\tilde D$.
\end{defi}
	Let $M$ be a rank one locally symmetric manifold such that its universal cover has volume entropy $2\ell$.
Let $\convex$ be a geodesically convex closed subset of $M$ with non empty interior and with smooth boundary. Let $S=\partial\convex$, and let $\nexp :S\times \bb R_{\geq 0}$ be the outer normal exponential map given by $\nexp (x,t)=\exp_x(t\normal(x))$. Assume that $\nexp $ is a diffeomorphism from $S\times\bb R_{\geq 0}$ to $M-\mathring{D}$.
By Lemma \ref{lem:normal}, the metric can be written as in \eqref{equ:gddt}. Let $h$ be the density function defined as in Proposition \ref{prop:product}.
\begin{lem}\label{volumefunction}
	With the above assumption, if we have an upper bound on the second fundamental form on $S$, then there exists $C_1>0$ depending on the bound such that the density function $h$ satisfies the following inequalities for every $x$ in $S$, and every $t\geq 0$:
	\begin{align}\label{equ:partialh}
	\partial_th(x,t)\geq 2\ell(\tanh t)h(x,t),\quad e^{2\ell t}/C_1\leq h(x,t)\leq C_1e^{2\ell t}.
	\end{align}
\end{lem}	
\begin{rem}
	The first inequality in \eqref{equ:partialh} has already been used in \cite[Lemma 2.3]{hamenstadt2004} without proof. For completeness, a proof is given here.
\end{rem}
\begin{rem}
	This lemma is a consequence of the negative curvature and the convexity. We do not have a better inequality $\partial_th(t)\geq 2\ell h(t)$. For example,  in $\bb H^2$, let $\convex$ be the $r-$neighbourhood of a geodesic, then we can compute explicitly that $h(x,t)=\cosh(t+r)/\cosh r$.
\end{rem}
% in \cite{petersen},page 44, he gives a equation relate the the derivative of volume with second fundamental form. But for our symmetric case, the result is optimal in some sense, so the explicit calculate is necessary.
%
%
%
The proofs of these two lemmas, which use the standard computations for Jacobi fields, will be given later. By Lemma \ref{volumefunction}, we have 
\begin{equation}\label{ineq:parth}
\partial_t h\geq 2\ell(\tanh t)h\geq 2\ell(1-2e^{-t} )h.
\end{equation}
Using Proposition \ref{prop:product} and Corollary \ref{cor:lpi} with $S\times \bb R_{\geq 0}$, (Proposition \ref{prop:product} deals with manifolds $L\times \bb R$, but the proof of $L\times \bb R_{\geq 0}$ is exactly the same) we have
\begin{lem}\label{prop:complement}
	Let $E$ be as in \eqref{equ:energy}. With the same assumptions as in Lemma \ref{volumefunction}, for a function $f\in C_c^{\infty}(\nexp(S\times \bb R_{\geq 0}))$, we have
	\begin{equation}\label{ineq:complement0}
	E(f)\geq \int (\|\nabla (e^{\ell t})f)\|^2e^{-2\ell t}-4\ell^2e^{-t}f^2)\dd vol.
	\end{equation}
Given $C>0$, there exist a compact set $I\subset\bb R_{\geq 0}$ and a constant $C_1>0$ such that for all compact set $K$ in $S$
	\begin{equation}\label{ineq:complement1} 
	\begin{split}
	E(f)+C_1\int_{\nexp(K\times I)}f^2\dd vol\geq C\int_{\nexp(K\times \bb R_{\geq 0})} e^{-t}f^2\dd vol.
	\end{split}
	\end{equation}
\end{lem}

It remains to prove Lemma \ref{lem:normal} and Lemma \ref{volumefunction}.
\begin{proof}[Proof of Lemma \ref{volumefunction}]
	The main idea is to compute the density function $h$ by the second fundamental form. The second fundamental form of $S$ at $x$ is defined to be the symmetric form $\two_S:T_xS\times T_xS\rightarrow \R$,
	\[\two_S(v,u)=g(D_v\normal(x),u), \]
	where $v,u$ are two vectors in the tangent space $T_xS$. By the convexity of $S$, the second fundamental form $\two_S$ is positive definite at every $x$ in $S$.
	
	Fix a point $x$ in $S$. By \eqref{equ:riecur}, starting from the outer unit normal vector $\normal(x)$, we can find an orthonormal basis $\{\normal(x), (Y_j)_{1\leq j\leq m}, (Y_{k})_{m+1\leq k\leq m+q}\}$ of $T_xM$ such that
	\begin{equation*}
	R(\normal(x),Y_j,\normal(x))=-Y_j,\ \ R(\normal(x),Y_{k},\normal(x))=-4Y_k, \text{ where }m=(n-1)\dim_{\bb R}\bb F, q=\dim_{\bb R}\bb F-1.
	\end{equation*}
	Let $B$ be the matrix representation of $\two_S$ with the basis $\{(Y_j)_{1\leq j\leq m}, (Y_{k})_{m+1\leq k\leq m+q}\}$ of $T_xS$.
	\begin{lem}\label{lem:at}
		With the same assumption as in Lemma \ref{volumefunction}, we have
		\begin{equation}\label{equ:at}
		h(x,t)=\det\left(\begin{pmatrix}
		\cosh t Id_{m} & 0 \\ 0 & \cosh 2t Id_{q}
		\end{pmatrix} + B\begin{pmatrix}
		\sinh t Id_{m} & 0 \\ 0 & \frac{1}{2}\sinh 2t Id_{q}
		\end{pmatrix} \right).
		\end{equation}
	\end{lem}
	\begin{proof}
	There exists a local chart on $S$ defined by ($\phi$, $U$) $\phi:\bb R^{m+q}\supset U\rightarrow S$ such that $\phi(0)=x$ and $\frac{\partial}{\partial u_i}\phi(0)=Y_i$. This particular choice of local chart implies that $\dd u$ is the volume element at $x$. Let $$\tilde{\phi}(u,t)=\exp_{\phi(u)}(t\normal(\phi(u))$$ 
	and $\tilde{U}=U\times \bb R_{\geq 0}$. Then $(\tilde{\phi},\tilde{U})$ is a local chart of $M$. For every fixed $u$, the curve $t\rightarrow\tilde{\phi}(u,t)$ is a geodesic starting from $\phi(u)$ with tangent vector $\normal(\phi(u))$.
	
	Fix all $u_w$ to $0$ except $w=i$. Then the map $H:\bb R^2\rightarrow M$ defined by $(u_i,t)\mapsto \tilde{\phi}(u_i,t)$, is a variation of geodesic. Let $J_i(t)$ be the Jacobi field defined by $J_i(t)=\frac{\partial}{\partial u_i}H(0,t)$. The volume element, in a local chart, can be written as $\sqrt{\det\left(g\left(\frac{\partial}{\partial x_i},\frac{\partial}{\partial x_w}\right)\right)}\dd x$. In our case, the volume element 
	at $\tilde{\phi}(0,t)$ is $\sqrt{\det(g(J_i(t),J_w(t))) }\dd u\dd t$. Therefore by definition, we have $h(x,t)=\sqrt{\det(g(J_i(t),J_w(t))) }$.
	
	Let $Y_i(t)$, $\normal(t)$ be the images of $Y_i$ and $\normal(x)$ under the parallel transport along the geodesic $t\mapsto\tilde{\phi}(0,t)$ with $t\geq 0$. They also form an orthonormal basis. By Lemma \ref{lem:normal}, the vectors $J_i(t)$ are orthogonal to $\normal(t)$. We decompose $J_i(t)$ with respect to $Y_i(t)$, that is $J_i(t)=\Sigma_w a_{iw}(t)Y_w(t)$, and write $A(t)=(a_{iw}(t))_{1\leq i,w\leq m+q}$. Then
	$A(0)=Id_{m+q}$ and the matrix $A'(0)$ equals $B$, the matrix representation of the second fundamental form. This is because by Schwarz's theorem, we have
	\begin{align*}
	a'_{iw}(0)=\partial_tg(J_i(t),Y_w(t))|_{t=0}=g\left(\frac{D}{\dd t}J_i(0),Y_w\right)=g\left(D_{\frac{\partial}{\partial  u_i}}\frac{\partial}{\partial t},Y_{w}\right)
	=g(D_{Y_i}\normal(x),Y_w).
	\end{align*}

	Because $J_i$ are Jacobi fields, they satisfy the Jacobi equation:
	\begin{equation}\label{equ:jacobi1}
	\frac{D^2}{\dd t^2}J_i(t)+R(\normal(t),J_i(t))\normal(t)=0.
	\end{equation}
	The map $R(\normal(t),\cdot)\normal(t)$ is a linear map on the orthogonal complement of $\normal(t)$ in the tangent space $T_{\tilde{\phi}(0,t)}M$. By \cite[IV, Thm1.3]{helgason1979differential}, in locally symmetric spaces the curvature tensor $R$ is invariant under parallel transport. Hence in our choice of the orthonormal basis $\{Y_i(t) \}_{1\leq i\leq m+q}$, the linear map $R(\normal(t),\cdot)\normal(t)$ can be represented by the matrix $diag\{Id_m,4Id_q\}$.
	From \eqref{equ:jacobi1}, we have
	\begin{equation*}
	A''(t)=A(t)\begin{pmatrix}
	Id_{m} & 0 \\ 0 & 4Id_{q}
	\end{pmatrix} .
	\end{equation*}
	The solution is determined by the initial conditions. Due to $A(0)=Id$ and $A'(0)=B$, it is
	\begin{equation*}
	A(t)=\begin{pmatrix}
	\cosh t Id_{m} & 0 \\ 0 & \cosh 2t Id_{q}
	\end{pmatrix} + B\begin{pmatrix}
	\sinh t Id_{m} & 0 \\ 0 & \frac{1}{2}\sinh 2t Id_{q}
	\end{pmatrix}.
	\end{equation*}
	Hence
	$h(x,t)=\det(g(J_i(t),J_w(t)))^{1/2}=\det(A(t))$, which implies the result.
\end{proof}

	For computing the determinant, we need a lemma.
	\begin{lem}\label{lem:determinant}
		Let $D$ be a diagonal matrix with nonnegative entries, let $B$ be a symmetric positive semidefinite matrix, and let $\begin{pmatrix}
		B_{11} & B_{12}\\ B_{21} & B_{22}
		\end{pmatrix}$ be a block partition of $B$ such that $B_{11}, B_{22}$ are square matrices. Then for all $\lambda_1,\lambda_2>0$ we have
		\[\det\left(D+\begin{pmatrix}
		\lambda_1B_{11} & \lambda_2B_{12}\\ \lambda_1B_{21} & \lambda_2B_{22}
		\end{pmatrix}\right)\geq \det(D).\]
	\end{lem}
	\begin{proof}
		It is elementary that the sum of two symmetric positive semidefinite matrices has determinant no less than the determinant of each one. We only need to transform our matrix to a symmetric matrix. We have
		\begin{align*}
		\det\left(D+\begin{pmatrix}
		\lambda_1B_{11} & \lambda_2B_{12}\\ \lambda_1B_{21} & \lambda_2B_{22}
		\end{pmatrix}\right)=\det\left((D\begin{pmatrix}
		\lambda_1^{-1} Id_m & 0\\ 0 &\lambda_2^{-1} Id_q
		\end{pmatrix}+\begin{pmatrix}
		B_{11} & B_{12}\\ B_{21} & B_{22}
		\end{pmatrix})\begin{pmatrix}
		\lambda_1 Id_m & 0\\ 0 &\lambda_2 Id_q
		\end{pmatrix}\right).
		\end{align*}
		Since $D$ is diagonal, the first matrix in the right-hand side is again symmetric. We have
		\begin{align*}
		\det\left(D+\begin{pmatrix}
		\lambda_1B_{11} & \lambda_2B_{12}\\ \lambda_1B_{21} & \lambda_2B_{22}
		\end{pmatrix}\right)\geq \det(D).
		\end{align*}
		The proof is complete.
	\end{proof}
	Return to the proof of Lemma \ref{volumefunction}.
	Let $\begin{pmatrix}
	B_{11} & B_{12}\\ B_{21} & B_{22}
	\end{pmatrix}$ be the block partition of $B$ such that $B_{11}, B_{22}$ are square matrices of order $m$, $q$. By \eqref{equ:at}
	\begin{equation}\label{equ:hxt}
	\begin{split}
	h(x,t)=(\cosh t)^m(\cosh 2t)^q\det\left(Id_{m+q}+\begin{pmatrix}
	\tanh tB_{11} & \frac{\tanh 2t}{2}B_{12}\\\tanh tB_{21}&\frac{\tanh 2t}{2} B_{22}
	\end{pmatrix}\right).
	\end{split}
	\end{equation}
	Since $B$ is positive semidefinite, using Lemma \ref{lem:determinant}, we have $h(t)\geq \cosh^mt\cosh^q2t\geq e^{2\ell t}/C_1$. The upper bound of $h(t)$ is due to the upper bound on the second fundamental form, that is to say $B$ is bounded. 
	
	The derivative in the scalar part of \eqref{equ:hxt}, $\cosh^mt\cosh^q2t$, gives us
	\begin{equation*}
	(m\tanh t+2q\tanh 2t)h(t)\geq (m+2q)\tanh t\,h(t)=2\ell\tanh t\,h(t).
	\end{equation*}
	It remains to prove the positivity of the derivative of the determinant part of \eqref{equ:hxt}, which is the sum of derivatives in every column. Since all the terms are similar, we need only to show that the derivative in the first column is non negative. The derivative of $\tanh t$ is $1/\cosh^2 t$, and we multiply the first column with $\tanh t\cosh^2 t$ to recover the original column. The determinant of the derivative of the first column becomes
	\begin{align*}
	\frac{1}{\tanh t\cosh^2t}\det\left(\begin{pmatrix}
	0&0\\
	0&Id_{m+q-1}
	\end{pmatrix}+\begin{pmatrix}
	\tanh tB_{11} & \frac{\tanh 2t}{2}B_{12}\\\tanh tB_{21}&\frac{\tanh 2t}{2} B_{22}
	\end{pmatrix}\right),
	\end{align*}
	which is nonnegative by Lemma \ref{lem:determinant}.
\end{proof}
It remains to prove Lemma \ref{lem:normal}.
\begin{proof}[Proof of Lemma \ref{lem:normal}]
	Use the same notation as in the proof of Lemma \ref{lem:at}. Let 
	$$J(t)=\frac{\partial}{\partial u_i}H(0,t)=\frac{\partial}{\partial u_i}\tilde{\phi}(u,t)|_{u=0}.$$
	Then $J(t)$ is a variation of geodesic, and it is a Jacobi field, which is determined by its value and derivative at $0$. We have $J(0)=\frac{\partial}{\partial u_i}\tilde\phi(u,0)|_{u=0}=\frac{\partial}{\partial u_i}\phi(0)$, which is a tangent vector of $S$ at $x$. Hence $J(0)$ is normal to $\normal(x)$. For the derivative, by Schwarz's theorem we have
	\begin{equation}\label{equ:jacobi}
	J'(0)=\frac{D}{\dd t}J(t)|_{t=0}=D_{\frac{\partial}{\partial  t}}\frac{\partial}{\partial u_i}=D_{\frac{\partial}{\partial  u_i}}\frac{\partial}{\partial t}. 
	\end{equation}
	Therefore
	\[g(J'(0),\normal(x))=g\left(D_{\frac{\partial}{\partial  u_i}}\frac{\partial}{\partial t},\frac{\partial}{\partial t}\right)=\frac{1}{2}\frac{\partial}{\partial u_i}\left\|\frac{\partial }{\partial t}\right\|^2=0. \]
	Hence $J(t)$ is normal to the tangent vector of the geodesic. This is true for all $i$, and the result follows.
\end{proof}
\section{Geometrically finite manifolds}
\label{sec:geofinman}
We return to the study of the spectrum of the whole manifold. We give a topological decomposition of geometrically finite manifolds. In order to describe it, we use standard cusp regions. For a subset $S$ of $M_c$, we use $\mathring{S}$ to denote its topological interior, that is $\mathring{S}=S-\bar{S^c}\cap S$. The reader who is only interested in convex cocompact manifolds (geometrically finite manifolds without cusps) can skip Section \ref{sec:stacus} and go directly to Section \ref{sec:goopar}. Our goal in this section is to obtain the Lax-Phillips inequality, then Theorem \ref{the:main} follows by Proposition \ref{prop:main}.
\subsection{Standard cusp regions}
\label{sec:stacus}
The manifold with boundary $M_c$ may have some cusps. To analyse the structure of cusps, we introduce the concept of a topological end. 
\begin{defi}\label{defi:end}
	Let $T$ be a differential manifold. An end $e$ is a function which assigns every compact subset $K$ of $T$ a non empty connected component of $T\backslash K$, such that if $K\subset K'$, then $e(K)\supset e(K')$. 
	
	Let $K_1\subset K_2\subset \dots$ be an ascending sequence of compact subsets whose interiors cover $T$. We call $e(K_i)$ a system of neighbourhoods for the end $e$. 
	
	A neighbourhood for the end $e$ is an open subset $U$ such that $U\supset e(K_n)$ for some $n$.
\end{defi}
\begin{prop}\cite[Proposition 4.4]{bowditch1995}\label{prop:end}
	If $ \Pi$ is a discrete parabolic subgroup of $G$, then $M_c( \Pi)$ has precisely one topological end. Moreover, we can find a system of neighbourhoods for the end consisting of geodesically convex submanifolds of $M_c( \Pi)$.
\end{prop}
Following \cite{bowditch1995}, we define a \textit{standard cusp region} (with boundary). The definition of Bowditch is not explicitly written, but it is implicitly defined in \cite[Sec.5.1]{bowditch1995}. For more details on the real hyperbolic case, we refer to \cite[Sec.3.1]{bowditch_geometrical_1993} and \cite[Page 6]{guillarmou2012resolvent}
\begin{defi}\label{defi:cusp region}
	Let $\Gamma$ be a discrete subgroup of $G$. A standard cusp region of $M_c(\Gamma)$ is a closed subset $\cusp$ which is isometric to a closed subset $E$ of $M_c(\Pi)$, where $\Pi$ is a discrete parabolic subgroup of $\Gamma$ and the interior of $E$ is supposed to be a geodesically convex submanifold and a neighbourhood for the end of $M_c(\Pi)$. 
\end{defi}
We have a corollary of Proposition \ref{prop:end}.

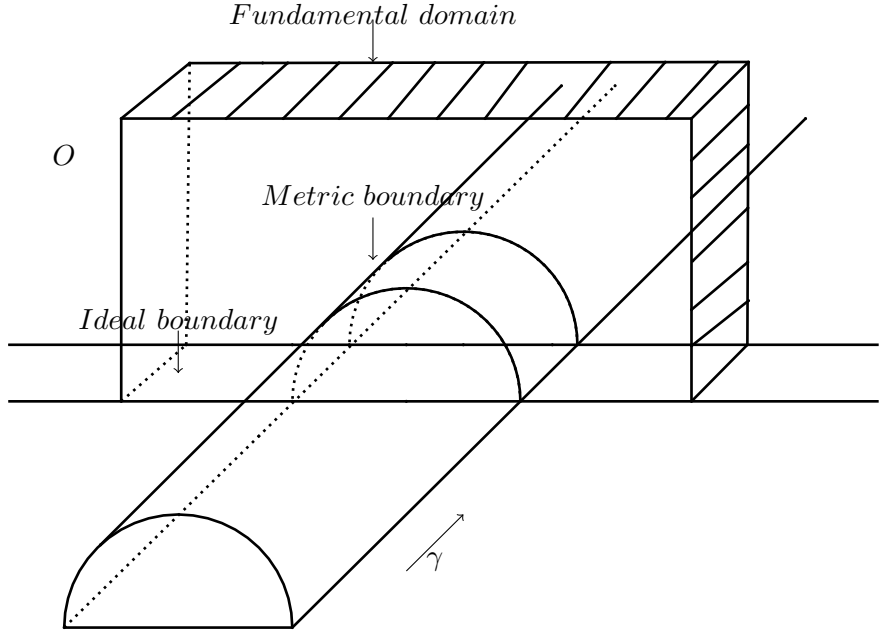
\begin{figure}\label{fig:stacus}
	\begin{center}
		\begin{tikzpicture}[scale =0.75]
		[line cap=round,line join=round,>=triangle 45,x=1cm,y=1cm]
		\clip(-4.981492655620974,-4.668612002520705) rectangle (10.282710898012887,7.982631013265871);\draw [line width=1pt,domain=-4.981492655620974:10.282710898012887] plot(\x,{(--2-0*\x)/2});\draw [line width=0.4pt,dotted,domain=-4.981492655620974:10.282710898012887] plot(\x,{(--2.56-0*\x)/2.56});\draw [shift={(2,0)},line width=1pt,dotted]  plot[domain=0:3.141592653589793,variable=\t]({1*2*cos(\t r)+0*2*sin(\t r)},{0*2*cos(\t r)+1*2*sin(\t r)});\draw [shift={(3,1)},line width=1pt,dotted]  plot[domain=0:3.141592653589793,variable=\t]({1*2*cos(\t r)+0*2*sin(\t r)},{0*2*cos(\t r)+1*2*sin(\t r)});\draw [line width=1pt] (-3,0)-- (-3,5);\draw [line width=1pt] (-3,5)-- (7,5);\draw [line width=1pt] (7,5)-- (7,0);\draw [shift={(3,1)},line width=1pt]  plot[domain=0:2.292451177659658,variable=\t]({1*2*cos(\t r)+0*2*sin(\t r)},{0*2*cos(\t r)+1*2*sin(\t r)});\draw [shift={(2,0)},line width=1pt]  plot[domain=0:2.390663591191853,variable=\t]({1*2*cos(\t r)+0*2*sin(\t r)},{0*2*cos(\t r)+1*2*sin(\t r)});\draw [line width=1pt] (0,-4)-- (9,5);\draw [line width=1pt,dotted] (-4,-4)-- (5.66,5.58);\draw [shift={(-2,-4)},line width=1pt]  plot[domain=0:3.141592653589793,variable=\t]({1*2*cos(\t r)+0*2*sin(\t r)},{0*2*cos(\t r)+1*2*sin(\t r)});\draw [line width=1pt] (-3.364636500720022,-2.537889463514262)-- (4.72,5.58);\draw [line width=1pt] (-4,-4)-- (0,-4);\draw [line width=1pt,dotted] (-1.86,1)-- (-1.8,5.98);\draw [line width=1pt] (7.98,1)-- (8,6);\draw [line width=1pt] (-1.8,5.98)-- (8,6);\draw [line width=1pt] (-3,5)-- (-1.8,5.98);\draw [line width=1pt] (7,5)-- (8,6);\draw [line width=1pt,dotted] (-3,0)-- (-1.86,1);\draw [line width=1pt] (7,0)-- (7.98,1);\draw [line width=1pt,domain=-4.981492655620974:10.282710898012887] plot(\x,{(-0-0*\x)/10});\draw [line width=1pt] (-2.12,5)-- (-0.9198812166546578,5.981796160782337);\draw [line width=1pt] (-1.16,5)-- (-0.07996634749542901,5.9835102727602125);\draw [line width=1pt] (-0.16,5)-- (0.8199482717689639,5.985346833207691);\draw [line width=1pt] (0.82,5)-- (1.760025989062936,5.987265359161352);\draw [line width=1pt] (1.68,5)-- (2.5001453554962287,5.988775806847951);\draw [line width=1pt] (2.54,5)-- (3.360100957513713,5.990530818280639);\draw [line width=1pt] (3.38,5)-- (4.119893711396454,5.992081415737543);\draw [line width=1pt] (4.78,5)-- (5.58005089524825,5.995061328357648);\draw [line width=1pt] (5.68,5)-- (6.539965264617806,5.99702033727473);\draw [line width=1pt] (6.56,5)-- (7.460002249053523,5.998897963773578);\draw [line width=1pt] (7,4.26)-- (7.996960688628982,5.2401721572454845);\draw [line width=1pt] (7,3.6)-- (7.994160733428266,4.540183357066287);\draw [line width=1pt] (7.02,2.48)-- (7.9896798451224775,3.41996128061951);\draw [line width=1pt] (7,1.62)-- (7.985840866546136,2.4602166365338154);\draw [line width=1pt] (7,0.98)-- (8,1.78);
		\begin{scriptsize}\draw [fill=xdxdff] (0,1) circle (0.5pt);\draw [fill=ududff] (2,1) circle (0.5pt);\draw [fill=xdxdff] (4.56,1) circle (0.5pt);\draw [fill=uuuuuu] (0,0) circle (0.5pt);\draw [fill=xdxdff] (1,1) circle (0.5pt);\draw [fill=xdxdff] (4,0) circle (0.5pt);\draw [fill=xdxdff] (5,1) circle (0.5pt);\draw [fill=xdxdff] (0.5378894635142621,1.3646365007200225) circle (0.5pt);\draw [fill=xdxdff] (1.6787442102159507,2.5014270338455105) circle (0.5pt);\draw [fill=xdxdff] (-3,0) circle (0.5pt);\draw [fill=ududff] (-3,5) circle (0.5pt);\draw [fill=ududff] (7,5) circle (0.5pt);\draw [fill=xdxdff] (7,0) circle (0.5pt);\draw [fill=xdxdff] (2,0) circle (0.5pt);\draw [fill=uuuuuu] (3,1) circle (0.5pt);\draw [fill=black] (0,-4) circle (0.5pt);\draw [fill=ududff] (9,5) circle (0.5pt);\draw [fill=black] (-4,-4) circle (0.5pt);\draw [fill=ududff] (5.66,5.58) circle (0.5pt);\draw [fill=xdxdff] (-3.364636500720022,-2.537889463514262) circle (0.5pt);\draw [fill=ududff] (4.72,5.58) circle (0.5pt);\draw [fill=xdxdff] (-0.5857548030168775,-0.6140301255591809) circle (0.5pt);\draw [fill=xdxdff] (-1.86,1) circle (0.5pt);\draw [fill=ududff] (-1.8,5.98) circle (0.5pt);\draw [fill=xdxdff] (7.98,1) circle (0.5pt);\draw [fill=ududff] (8,6) circle (0.5pt);\draw [fill=black] (-2.12,5) circle (0.5pt);\draw [fill=black] (-0.5202094118725038,5.9826118175267915) circle (0.5pt);\draw [fill=black] (-0.9198812166546578,5.981796160782337) circle (0.5pt);\draw [fill=black] (-1.16,5) circle (0.5pt);\draw [fill=black] (-0.07996634749542901,5.9835102727602125) circle (0.5pt);\draw [fill=black] (-0.16,5) circle (0.5pt);\draw [fill=black] (0.8199482717689639,5.985346833207691) circle (0.5pt);\draw [fill=black] (0.82,5) circle (0.5pt);\draw [fill=black] (1.760025989062936,5.987265359161352) circle (0.5pt);\draw [fill=black] (1.68,5) circle (0.5pt);\draw [fill=black] (2.5001453554962287,5.988775806847951) circle (0.5pt);\draw [fill=black] (2.54,5) circle (0.5pt);\draw [fill=black] (3.360100957513713,5.990530818280639) circle (0.5pt);\draw [fill=black] (3.38,5) circle (0.5pt);\draw [fill=black] (4.119893711396454,5.992081415737543) circle (0.5pt);\draw [fill=black] (4.78,5) circle (0.5pt);\draw [fill=black] (5.58005089524825,5.995061328357648) circle (0.5pt);\draw [fill=black] (5.68,5) circle (0.5pt);\draw [fill=black] (6.539965264617806,5.99702033727473) circle (0.5pt);\draw [fill=black] (6.56,5) circle (0.5pt);\draw [fill=black] (7.460002249053523,5.998897963773578) circle (0.5pt);\draw [fill=black] (7,4.26) circle (0.5pt);\draw [fill=black] (7.996960688628982,5.2401721572454845) circle (0.5pt);\draw [fill=black] (7,3.6) circle (0.5pt);\draw [fill=black] (7.994160733428266,4.540183357066287) circle (0.5pt);\draw [fill=black] (7.02,2.48) circle (0.5pt);\draw [fill=black] (7.9896798451224775,3.41996128061951) circle (0.5pt);\draw [fill=black] (7,1.62) circle (0.5pt);\draw [fill=black] (7.985840866546136,2.4602166365338154) circle (0.5pt);\draw [fill=black] (7,0.98) circle (0.5pt);\draw [fill=black] (8,1.78) circle (0.5pt);\end{scriptsize}
		\draw ({sqrt(2)},6.5) node[above]{$Fundamental\ domain$};
		\draw [<-] ({sqrt(2)},6) -- ({sqrt(2)},6.75);
		\draw(-4,4) node[above]{$\cusp$};
		\draw (-2,1) node[above]{$Ideal\ boundary$};
		\draw [->] (-2,1.25) -- (-2,0.5);
		\draw ({sqrt(2)},4) node[below] {$Metric\  boundary$};
		\draw [->] ({sqrt(2)},3.25) -- ({sqrt(2)},2.5);
		\draw (2.5,-2.5) node[below]{$\gamma$};
		\draw [->] (2,-3) -- (3,-2);
		\end{tikzpicture}
	\end{center}
	\caption{A standard cusp region of rank one}
\end{figure}
\begin{prop}\label{prop:cuspregion} 
	Let $\Gamma$ be a discrete subgroup of $G$, and let $\cusp$ be a standard cusp region of $M_c(\Gamma)$. There exists a smaller standard cusp region $\cusp'$ such that $\cusp'$ is contained in the interior of $\cusp$.
\end{prop}
\begin{proof}
	By Definition \ref{defi:cusp region}, the cusp region $\cusp$ is isometric to $E$ and the complement of the interior of $E$ is a compact subset $K$ of $M_c(\Pi)$. Denote by $f$ the isometric map from $\cusp$ to $E$. By Proposition \ref{prop:end}, $M_c(\Pi)$ has only one end and we have a larger compact subset $K_1$, whose interior contains $K$ and whose complement is geodesically convex. Let $E_1=\overline{K_1^c}$ and let $\cusp'$ be the corresponding subset of $\cusp$, that is $\cusp'=f^{-1}(E_1)$.
\end{proof}
In order to give more intuitive, we add some explications.
For a symmetric space of real rank one, we use the half space model introduced at the end of Section \ref{sec:rankone}, that is a diffeomorphism from $N\times \R_{>0}$ to $X$. Recall that $\Pi$ is a parabolic subgroup of $G$ which fixes $\infty$. It is a subgroup of $\M\ltimes N$, which acts isometrically on $N$ equipped with the metric $g_t$. We have
\[M_c(\Pi)\simeq(\Pi\backslash N)\times \bb R_{\geq 0} \text{ as topological spaces} \]
Then our standard cusp region is isometric to the complement of an open relatively compact subset in $M_c(\Pi)$, with additional convex condition on the interior. (see Figure \ref{fig:stacus}.)

Now, we give a description of the standard cusp region in maximal rank case. This will be used in the proof of the main theorem for geometrically finite manifolds with only maximal rank cusps. The main idea is the same for manifolds with only maximal ranks cusps and the general cases.

\begin{defi}[Maximal rank]\label{def:maxran}
	Let $\Pi$ be a discrete parabolic subgroup of $G$. We call $\Pi$ a subgroup of maximal rank if the quotient $\Pi\backslash N$ is compact.
	
	A cusp region is said to be maximal rank, if the corresponding discrete parabolic subgroup is of maximal rank
\end{defi}
\begin{rem}
	We explain our definition as follows. The rank of a nilpotent group is defined to the sum of the rank of its central series.
	
	For real hyperbolic case, $\Pi$ is a discrete subgroup of $Isom(\bb R^{n-1})=O(n-1)\ltimes\bb R^{n-1}$. By Bieberbach's theorem (see for example \cite[Theorem 2.2.5]{bowditch_geometrical_1993}), the group $\Pi$ is virtually abelian. The rank of $\Pi$ is defined to be the rank of its maximal normal abelian subgroup. Hence when $\Pi$ attends maximal rank, the quotient space is compact.
	
	For rank one symmetric spaces, by Margulis' Lemma, the discrete parabolic subgroup $\Pi$ is virtually nilpotent. As in \cite[Lemma 3.4]{corletteIozzi1999}, for a virtually nilpotent discrete subgroup $\Pi<\M\ltimes N$, we can find a subgroup $\Pi_1<\Pi$ of finite index which is nilpotent, and there exists a subgroup $\Pi_2< N$ which is isomorphic to $\Pi_1$ and satisfies $\Pi_1\cdot x=\Pi_2\cdot x$ for some $x$ in $N$. This means that the $\Pi_1$, $\Pi_2$-orbits of $x$ are the same in $N$. Let $N_2$ be the Zariski closure of $\Pi_2$. Then the rank of $\Pi$ is the same as the rank of $N_2$. When $\Pi$ attends maximal rank, $N_2$ coincides with $N$. Then $\Pi_2$ is a cocompact subgroup of $N$ because $N$ is nilpotent. Due to $\Pi_2\cdot x=\Pi_1\cdot x$, every point in $N$ has a bounded distance to the orbit $\Pi_1\cdot x$. Hence $\Pi_1\backslash N$ is compact, so is $\Pi\backslash N$. 
\end{rem}
\begin{prop}[Maximal rank version of Proposition \ref{prop:cuspregion}]\label{prop:mrank}
	Let $\Gamma$ be a discrete subgroup of $G$.  Let $\cusp$ be a standard cusp region of $M_c(\Gamma)$ with maximal rank. Then we can find a smaller cusp region $\cusp_1$, which is isomorphic to the quotient of a horoball by a discrete parabolic subgroup of $\Gamma$.
\end{prop}
\begin{proof}
	By Definition \ref{defi:cusp region}, there is a discrete subgroup $\Pi$ of $\Gamma$, such that $\cusp$ is isometric to a closed subset $E$ of $M_c(\Pi)$, whose interior is a neighbourhood for the end of $M_c(\Pi)$. By Definition \ref{defi:end} and Proposition \ref{prop:end}, the complement of $E$ in $M_c(\Pi)$ is relatively compact. Under the half space model, we can suppose that $E^c\subset (\Pi\backslash N)\times[0,1]$.
	Let $B$ be the horoball, which is homeomorphic to $N\times \R_{\geq 1}$. Then the quotient $\Pi\backslash B$ is a subset of $E$. Due to maximal rank, the quotient $\Pi\backslash N$ is compact. Hence
	\[M_c(\Pi)-\Pi\backslash\mathring{ B}\simeq(\Pi\backslash N)\times [0,1] \]
	is compact. The quotient $\Pi\backslash B$ is geodesically convex. Let $\cusp_1$ be the preimage of $\Pi\backslash B$ in $\cusp$ under the isometric map from $\cusp$ to $E$. The proof is complete.
\end{proof}

In later proof, for cusps of maximal rank, we will always take the quotient of horoball as a standard cusp region.
\begin{rem}
	If the cusp region is not of maximal rank, then the quotient of a horoball by the corresponding parabolic group is not a neighbourhood of the parabolic end in the sense of Definition \ref{defi:end}. Therefore it cannot by taken as a standard cusp region, which has to be much larger.  See the example of a standard cusp region for a rank 1 parabolic end in $\bb H^3$ given in Figure \ref{fig:stacus}.
\end{rem}
\subsection{A good partition of unity}
\label{sec:goopar}
\begin{defi}[Geometrical finiteness]\label{def:geofinite}
	A discrete subgroup $\Gamma$ in $G$ is called geometrically finite, if $M_c(\Gamma)$ is the union of a compact set and a finite number of standard cusp regions $\cusp_i$ for $1\leq i\leq k$, that is to say $M_c(\Gamma)-\bigcup_{1\leq i\leq k} \mathring{\cusp_i}$ is compact.
\end{defi}
This definition is not explicitly written in \cite{bowditch1995}, but is given in the discussion after \cite[Def. F1]{bowditch1995}. (See also \cite[Def.(GF1)]{bowditch_geometrical_1993} for the real hyperbolic case. In \cite{bowditch_geometrical_1993}, Bowditch explained the equivalence of the definition in the introduction and Definition \ref{def:geofinite} for the real hyperbolic case.) By \cite[Lemma 6.2]{bowditch1995}, if $\Gamma$ is geometrically finite, then there exist standard cusp regions $\cusp_i$ for $1\leq i\leq k$, such that $\cusp_i$ are pairwise disjoint. For the purpose of the exposition, we can limit our consideration to the case that there is at most one end; the results hold and the methods of proof work for general cases.

For a real number $r>0$, we define the $r$-neighbourhood of a set $Q$ in $X$ by  $N_r(Q)=\{x\in X|d(x,Q)\leq r\}$. Let $W_r=N_r(\rm{hull}(\Lambda(\Gamma))\cap X)$ be the $r$-neighbourhood of the convex hull of the limit set $\Lambda(\Gamma)$. Let $C(M)$ be the convex core defined by $C(M)=\Gamma\backslash( \rm{hull}(\Lambda(\Gamma))\cap X)$. Let $C_{r}(M)=\Gamma\backslash W_r$ be the $r$-neighbourhood of the convex core. One problem here is that the boundary of $C_{r}(M)$ may not be $C^\infty$-smooth, but is only $C^{1,1}$-smooth (see for instance \cite{walter76} or \cite{federer1959}). To overcome this difficulty, we use a result of \cite[Prop.6]{parkkonenpaulin2012}.
(In the statement of Proposition 6 in \cite{parkkonenpaulin2012}, they do not have a $\Gamma$-invariant condition. But if we start from a $\Gamma$-invariant set, their method automatically gives us a $\Gamma$-invariant set.) We can find a closed convex subset  $W'$ with $C^\infty$ smooth boundary such that $W_1\subset W'\subset W_{3/2}$ and $W'$ is also $\Gamma$-invariant. Let $\convex=\Gamma\backslash W'$. Then $\convex\subset \mathring{C_2(M)}$.

Let $\cusp$ be a standard cusp region of the unique cusp in $M_c$. We have a smaller standard cusp region $\cusp'$ such that $\cusp'\subset\mathring{\cusp}$.

We can cover the geometrically finite manifold with three open sets
\begin{equation}
M_c=\mathring{\cusp}\cup(\convex^c-\cusp')\cup(\mathring{C_{2}(M)}-\cusp') ,
\end{equation}
where $\convex^c$ is the complement in $M_c$. For the simplicity of the notation, we write 
$$M_1=\mathring{\cusp}, M_2=\convex^c-\cusp' \text{ and } M_3=\mathring{C_{2}(M)}-\cusp'.$$ 
Since $M_c$ inherits the differential structure from $X\cup \Omega$,  the covering is about a differential manifold with boundary. We can find a smooth partition of unity subordinate to this cover, written as $\{\bar{\varphi_1},\  \bar{\varphi_2},\ \bar{\varphi_3}\}$, which is smooth on the boundary. Here $M_1,M_2$ may intersect the ideal boundary $M_I=M_c- M$.
\definecolor{qqqqqq}{rgb}{0,0,0}
\begin{figure}
	\begin{center}
\begin{tikzpicture}
[line cap=round,line join=round,>=triangle 45,x=1cm,y=1cm]\clip(-10.61738836520274,-4.468047622401873) rectangle (6.437457615263054,4.322335548146636);\draw [rotate around={-43.70783311375007:(-2.952139588352604,-2.7726201372157995)},line width=1pt] (-2.952139588352604,-2.7726201372157995) ellipse (1.1085578566372463cm and 0.5865155765298693cm);\draw [line width=1pt] (3.9,0.7)--  (3.9,0.7)-- (3.78,0.7)-- (3.66,0.7)-- (3.52,0.7)-- (3.44,0.7)-- (3.34,0.7)-- (3.22,0.7)-- (3.14,0.7)-- (3.02,0.7)-- (2.92,0.7)-- (2.84,0.7)-- (2.74,0.7)-- (2.64,0.7)-- (2.56,0.7)-- (2.48,0.7)-- (2.4,0.7)-- (2.32,0.72)-- (2.24,0.72)-- (2.16,0.72)-- (2.06,0.74)-- (1.92,0.74)-- (1.84,0.74)-- (1.74,0.76)-- (1.62,0.76)-- (1.56,0.78)-- (1.46,0.78)-- (1.32,0.8)-- (1.24,0.82)-- (1.14,0.84)-- (1.04,0.84)-- (0.96,0.86)-- (0.88,0.88)-- (0.82,0.9)-- (0.72,0.92)-- (0.66,0.94)-- (0.58,0.96)-- (0.52,0.98)-- (0.46,1)-- (0.4,1.02)-- (0.34,1.04)-- (0.28,1.08)-- (0.22,1.1)-- (0.14,1.14)-- (0.06,1.22)-- (0,1.26)-- (-0.04,1.32)-- (-0.1,1.38)-- (-0.14,1.44)-- (-0.18,1.5)-- (-0.24,1.56)-- (-0.28,1.62)-- (-0.3,1.68)-- (-0.32,1.74)-- (-0.34,1.8)-- (-0.38,1.86)-- (-0.38,1.94)-- (-0.4,2)-- (-0.42,2.06)-- (-0.42,2.14)-- (-0.42,2.22)-- (-0.42,2.3)-- (-0.42,2.38)-- (-0.42,2.46)-- (-0.4,2.52)-- (-0.38,2.58)-- (-0.36,2.64)-- (-0.32,2.7)-- (-0.3,2.76)-- (-0.24,2.82)-- (-0.18,2.86)-- (-0.12,2.9)-- (-0.06,2.94)-- (0,2.98)-- (0.06,3)-- (0.14,3.04)-- (0.2,3.06)-- (0.28,3.08)-- (0.36,3.08)-- (0.44,3.08)-- (0.52,3.08)-- (0.6,3.08)-- (0.7,3.06)-- (0.78,3.04)-- (0.84,3.02)-- (0.92,2.98)-- (0.98,2.96)-- (1.04,2.92)-- (1.1,2.88)-- (1.16,2.84)-- (1.2,2.78)-- (1.24,2.72)-- (1.28,2.66)-- (1.32,2.6)-- (1.34,2.54)-- (1.38,2.48)-- (1.4,2.42)-- (1.44,2.36)-- (1.48,2.3)-- (1.52,2.24)-- (1.58,2.18)-- (1.62,2.12)-- (1.68,2.08)-- (1.74,2.04)-- (1.8,2)-- (1.86,1.96)-- (1.92,1.92)-- (1.98,1.88)-- (2.04,1.84)-- (2.1,1.8)-- (2.16,1.76)-- (2.22,1.74)-- (2.28,1.7)-- (2.34,1.66)-- (2.4,1.62)-- (2.46,1.58)-- (2.52,1.54)-- (2.58,1.5)-- (2.64,1.48)-- (2.7,1.44)-- (2.78,1.4)-- (2.84,1.38)-- (2.9,1.34)-- (2.96,1.32)-- (3.02,1.3)-- (3.08,1.26)-- (3.14,1.24)-- (3.2,1.22)-- (3.26,1.2)-- (3.32,1.16)-- (3.38,1.14)-- (3.44,1.12)-- (3.5,1.08)-- (3.56,1.06)-- (3.62,1.02)-- (3.68,1)-- (3.74,0.96)-- (3.8,0.94)-- (3.86,0.9)-- (3.92,0.88)-- (3.98,0.84)-- (4.04,0.82)-- (4.1,0.78)-- (4.16,0.76)-- (4.22,0.72)-- (4.28,0.68)-- (4.34,0.66)-- (4.4,0.62);\draw [line width=1pt] (3.82,0.7)--  (3.82,0.7)-- (3.9,0.7)-- (3.98,0.7)-- (4.06,0.7)-- (4.12,0.68)-- (4.2,0.68)-- (4.26,0.66)-- (4.32,0.64)-- (4.38,0.62)-- (4.44,0.58)-- (4.5,0.56);\draw [line width=1pt] (-3.7,-1.96)--  (-3.7,-1.96)-- (-3.7,-1.84)-- (-3.7,-1.76)-- (-3.7,-1.66)-- (-3.7,-1.54)-- (-3.7,-1.4)-- (-3.7,-1.28)-- (-3.7,-1.14)-- (-3.7,-1)-- (-3.7,-0.88)-- (-3.72,-0.76)-- (-3.74,-0.62)-- (-3.78,-0.52)-- (-3.82,-0.38)-- (-3.86,-0.26)-- (-3.94,-0.12)-- (-4,0)-- (-4.08,0.12)-- (-4.14,0.22)-- (-4.22,0.34)-- (-4.3,0.44)-- (-4.36,0.54)-- (-4.46,0.64)-- (-4.52,0.72)-- (-4.62,0.82)-- (-4.72,0.9)-- (-4.8,0.98)-- (-4.88,1.08)-- (-4.96,1.16)-- (-5.04,1.22)-- (-5.1,1.3)-- (-5.18,1.36)-- (-5.26,1.44)-- (-5.36,1.5)-- (-5.44,1.56)-- (-5.52,1.64)-- (-5.62,1.7)-- (-5.72,1.76)-- (-5.82,1.84)-- (-5.94,1.9)-- (-6.02,1.98)-- (-6.12,2.04)-- (-6.22,2.1)-- (-6.3,2.16)-- (-6.4,2.22)-- (-6.5,2.26)-- (-6.6,2.32)-- (-6.68,2.36)-- (-6.78,2.4)-- (-6.86,2.44)-- (-6.92,2.46)-- (-7.06,2.52)-- (-7.14,2.54)-- (-7.22,2.56)-- (-7.32,2.6)-- (-7.44,2.62)-- (-7.52,2.62)-- (-7.6,2.62)-- (-7.68,2.62)-- (-7.76,2.62)-- (-7.84,2.62)-- (-7.76,2.62)-- (-7.68,2.62)-- (-7.6,2.62)-- (-7.52,2.62)-- (-7.44,2.62)-- (-7.36,2.62)-- (-7.28,2.62)-- (-7.22,2.6)-- (-7.14,2.6)-- (-7.04,2.6)-- (-6.94,2.6)-- (-6.84,2.6)-- (-6.76,2.6)-- (-6.68,2.6)-- (-6.62,2.58)-- (-6.52,2.58)-- (-6.42,2.58)-- (-6.34,2.58)-- (-6.28,2.56)-- (-6.2,2.56)-- (-6.12,2.56)-- (-6.02,2.56)-- (-5.96,2.54)-- (-5.86,2.54)-- (-5.78,2.52)-- (-5.68,2.52)-- (-5.6,2.5)-- (-5.52,2.5)-- (-5.42,2.48)-- (-5.32,2.48)-- (-5.22,2.48)-- (-5.16,2.46)-- (-5.08,2.46)-- (-5,2.46)-- (-4.94,2.44)-- (-4.86,2.44)-- (-4.8,2.42)-- (-4.72,2.42)-- (-4.64,2.42)-- (-4.56,2.42)-- (-4.5,2.4)-- (-4.42,2.4)-- (-4.34,2.4)-- (-4.28,2.38)-- (-4.2,2.38)-- (-4.12,2.38)-- (-4.02,2.38)-- (-3.96,2.36)-- (-3.88,2.36)-- (-3.82,2.34)-- (-3.76,2.32)-- (-3.68,2.3)-- (-3.6,2.3)-- (-3.52,2.28)-- (-3.44,2.26)-- (-3.36,2.24)-- (-3.28,2.24)-- (-3.2,2.22)-- (-3.12,2.2)-- (-3.04,2.18)-- (-2.96,2.18)-- (-2.9,2.16)-- (-2.82,2.14)-- (-2.74,2.14)-- (-2.66,2.14)-- (-2.58,2.14)-- (-2.5,2.14)-- (-2.42,2.14)-- (-2.34,2.14)-- (-2.26,2.14)-- (-2.18,2.14)-- (-2.1,2.14)-- (-2.02,2.14)-- (-1.94,2.14)-- (-1.86,2.16)-- (-1.78,2.16)-- (-1.72,2.18)-- (-1.64,2.18)-- (-1.58,2.2)-- (-1.52,2.22)-- (-1.44,2.22)-- (-1.38,2.24)-- (-1.32,2.26)-- (-1.24,2.28)-- (-1.18,2.3)-- (-1.1,2.3)-- (-1.04,2.34)-- (-0.98,2.36)-- (-0.9,2.4)-- (-0.84,2.42)-- (-0.76,2.44)-- (-0.7,2.46)-- (-0.64,2.48)-- (-0.58,2.5)-- (-0.52,2.54)-- (-0.46,2.56)-- (-0.4,2.6)-- (-0.36,2.66)-- (-0.32,2.72)-- (-0.28,2.78);\draw [rotate around={59.76110146492801:(-4.820349201932394,1.861481227852139)},line width=1pt, dotted] (-4.820349201932394,1.861481227852139) ellipse (0.6229664764385643cm and 0.37382715222752294cm);\draw [line width=1pt] (4.44,0.56)--  (4.44,0.56)-- (4.36,0.56)-- (4.28,0.56)-- (4.18,0.56)-- (4.1,0.54)-- (4.02,0.52)-- (3.92,0.52)-- (3.86,0.5)-- (3.76,0.48)-- (3.68,0.44)-- (3.62,0.42)-- (3.54,0.4)-- (3.48,0.38)-- (3.42,0.36)-- (3.36,0.34)-- (3.28,0.32)-- (3.2,0.3)-- (3.1,0.28)-- (3.04,0.26)-- (2.98,0.24)-- (2.92,0.22)-- (2.86,0.2)-- (2.76,0.18)-- (2.7,0.16)-- (2.64,0.14)-- (2.58,0.12)-- (2.52,0.1)-- (2.46,0.08)-- (2.4,0.06)-- (2.34,0.02)-- (2.28,-0.02)-- (2.22,-0.06)-- (2.18,-0.12)-- (2.12,-0.18)-- (2.06,-0.24)-- (2.02,-0.3)-- (1.96,-0.34)-- (1.94,-0.4)-- (1.9,-0.46)-- (1.88,-0.52)-- (1.86,-0.58)-- (1.82,-0.64)-- (1.8,-0.72)-- (1.78,-0.78)-- (1.76,-0.84)-- (1.74,-0.92)-- (1.72,-0.98)-- (1.72,-1.06)-- (1.72,-1.14)-- (1.72,-1.22)-- (1.72,-1.3)-- (1.72,-1.38)-- (1.72,-1.46)-- (1.72,-1.54)-- (1.72,-1.62)-- (1.74,-1.68)-- (1.76,-1.74)-- (1.76,-1.84)-- (1.78,-1.9)-- (1.78,-1.98)-- (1.8,-2.04)-- (1.8,-2.12)-- (1.82,-2.18)-- (1.84,-2.24)-- (1.86,-2.3)-- (1.88,-2.36)-- (1.94,-2.38)-- (2.02,-2.38)-- (2.08,-2.34)-- (2.14,-2.32)-- (2.2,-2.28)-- (2.26,-2.26)-- (2.32,-2.22)-- (2.38,-2.2)-- (2.44,-2.16)-- (2.48,-2.1)-- (2.52,-2.04)-- (2.54,-1.98)-- (2.56,-1.92)-- (2.6,-1.86)-- (2.62,-1.78)-- (2.64,-1.72)-- (2.66,-1.66)-- (2.68,-1.58)-- (2.7,-1.52)-- (2.72,-1.44)-- (2.74,-1.38)-- (2.78,-1.3)-- (2.8,-1.24)-- (2.82,-1.18)-- (2.84,-1.12)-- (2.88,-1.06)-- (2.9,-1)-- (2.92,-0.94)-- (2.96,-0.88)-- (2.98,-0.82)-- (3.02,-0.76)-- (3.04,-0.7)-- (3.08,-0.64)-- (3.1,-0.58)-- (3.12,-0.52)-- (3.16,-0.46)-- (3.18,-0.4)-- (3.22,-0.34)-- (3.24,-0.28)-- (3.28,-0.22)-- (3.3,-0.16)-- (3.36,-0.1)-- (3.38,-0.04)-- (3.42,0.02)-- (3.46,0.08)-- (3.52,0.12)-- (3.56,0.18)-- (3.62,0.24)-- (3.68,0.28)-- (3.74,0.32)-- (3.8,0.36)-- (3.86,0.38)-- (3.92,0.42)-- (3.98,0.44)-- (4.04,0.48)-- (4.1,0.5)-- (4.16,0.52)-- (4.22,0.54)-- (4.28,0.56)-- (4.34,0.58)-- (4.4,0.6)-- (4.46,0.62)-- (4.52,0.64);\draw [line width=1pt] (-2.12,-3.54)--  (-2.12,-3.54)-- (-2.1,-3.46)-- (-2.04,-3.36)-- (-1.96,-3.24)-- (-1.92,-3.16)-- (-1.88,-3.1)-- (-1.84,-3.04)-- (-1.78,-2.94)-- (-1.72,-2.88)-- (-1.66,-2.82)-- (-1.62,-2.76)-- (-1.56,-2.7)-- (-1.5,-2.64)-- (-1.44,-2.58)-- (-1.38,-2.54)-- (-1.34,-2.48)-- (-1.28,-2.44)-- (-1.24,-2.38)-- (-1.18,-2.32)-- (-1.12,-2.26)-- (-1.06,-2.22)-- (-0.98,-2.14)-- (-0.92,-2.1)-- (-0.84,-2.06)-- (-0.76,-2.02)-- (-0.7,-2)-- (-0.64,-1.96)-- (-0.56,-1.92)-- (-0.5,-1.88)-- (-0.44,-1.86)-- (-0.38,-1.82)-- (-0.32,-1.8)-- (-0.26,-1.76)-- (-0.2,-1.74)-- (-0.14,-1.72)-- (-0.08,-1.7)-- (-0.02,-1.68)-- (0.06,-1.68)-- (0.14,-1.68)-- (0.22,-1.68)-- (0.3,-1.68)-- (0.38,-1.68)-- (0.46,-1.68)-- (0.54,-1.68)-- (0.6,-1.7)-- (0.66,-1.72)-- (0.72,-1.74)-- (0.78,-1.76)-- (0.84,-1.8)-- (0.9,-1.82)-- (0.96,-1.86)-- (1.02,-1.9)-- (1.08,-1.94)-- (1.14,-1.96)-- (1.2,-1.98)-- (1.24,-2.04)-- (1.3,-2.06)-- (1.38,-2.08)-- (1.44,-2.1)-- (1.5,-2.12)-- (1.56,-2.14)-- (1.62,-2.16)-- (1.68,-2.18)-- (1.74,-2.2)-- (1.8,-2.22)-- (1.86,-2.24)-- (1.88,-2.3)-- (1.9,-2.36);\draw [line width=1pt,dotted] (-4.690013224666891,2.134023137647478)--  (-4.690013224666891,2.134023137647478)-- (-4.68561875745547,2.120839736013216)-- (-4.68561875745547,2.103261867167533)-- (-4.6812242902440495,2.090078465533271)-- (-4.6812242902440495,2.0725005966875885)-- (-4.676829823032628,2.0593171950533264)-- (-4.676829823032628,2.0417393262076438)-- (-4.672435355821208,2.0285559245733817)-- (-4.672435355821208,2.0065835885162784)-- (-4.672435355821208,1.9890057196705957)-- (-4.672435355821208,1.9714278508249132)-- (-4.672435355821208,1.9538499819792305)-- (-4.672435355821208,1.9362721131335479)-- (-4.672435355821208,1.9186942442878652)-- (-4.676829823032628,1.9011163754421825)-- (-4.676829823032628,1.8835385065964998)-- (-4.6812242902440495,1.8703551049622378)-- (-4.68561875745547,1.8571717033279758)-- (-4.690013224666891,1.8439883016937137)-- (-4.694407691878311,1.830804900059452)-- (-4.698802159089732,1.81762149842519)-- (-4.703196626301152,1.804438096790928)-- (-4.711985560723994,1.791254695156666)-- (-4.725168962358256,1.778071293522404)-- (-4.738352363992518,1.7692823590995626)-- (-4.747141298415359,1.7560989574653005)-- (-4.760324700049622,1.7517044902538799)-- (-4.773508101683883,1.7385210886196178)-- (-4.786691503318146,1.7297321541967765)-- (-4.799874904952407,1.7165487525625145)-- (-4.81305830658667,1.7077598181396731)-- (-4.826241708220931,1.6989708837168318)-- (-4.839425109855194,1.6901819492939907)-- (-4.852608511489455,1.6769985476597287)-- (-4.865791913123718,1.672604080448308)-- (-4.878975314757979,1.6638151460254667)-- (-4.892158716392242,1.659420678814046)-- (-4.905342118026503,1.6550262116026253)-- (-4.918525519660766,1.6506317443912046)-- (-4.931708921295027,1.646237277179784)-- (-4.94928679014071,1.646237277179784)-- (-4.966864658986393,1.646237277179784)-- (-4.980048060620655,1.6418428099683633)-- (-4.997625929466338,1.6418428099683633)-- (-5.015203798312021,1.6418428099683633)-- (-5.032781667157703,1.6418428099683633)-- (-5.050359536003386,1.6418428099683633);\draw [line width=1pt, dotted] (-4.94928679014071,1.6506317443912046)--  (-4.94928679014071,1.6506317443912046)-- (-4.94928679014071,1.672604080448308)-- (-4.94928679014071,1.6901819492939907)-- (-4.94928679014071,1.7077598181396731)-- (-4.94928679014071,1.7297321541967765)-- (-4.94489232292929,1.7473100230424592)-- (-4.940497855717869,1.7604934246767212)-- (-4.9361033885064485,1.7736768263109832)-- (-4.931708921295027,1.791254695156666)-- (-4.927314454083607,1.804438096790928)-- (-4.922919986872186,1.8220159656366106)-- (-4.918525519660766,1.8351993672708724)-- (-4.914131052449345,1.8483827689051344)-- (-4.9097365852379244,1.8615661705393964)-- (-4.900947650815083,1.8747495721736585)-- (-4.892158716392242,1.8879329738079205)-- (-4.8833697819694,1.9011163754421825)-- (-4.874580847546559,1.9142997770764445)-- (-4.865791913123718,1.9274831787107065)-- (-4.852608511489455,1.9362721131335479)-- (-4.839425109855194,1.9450610475563892)-- (-4.826241708220931,1.9538499819792305)-- (-4.81305830658667,1.9626389164020719)-- (-4.799874904952407,1.9714278508249132)-- (-4.786691503318146,1.9758223180363337)-- (-4.773508101683883,1.9802167852477544)-- (-4.760324700049622,1.9890057196705957)-- (-4.747141298415359,1.997794654093437)-- (-4.733957896781098,2.0021891213048577)-- (-4.716380027935415,2.0065835885162784)-- (-4.698802159089732,2.0065835885162784)-- (-4.68561875745547,2.010978055727699)-- (-4.668040888609787,2.010978055727699)-- (-4.650463019764104,2.010978055727699)-- (-4.663646421398367,2.0065835885162784);\draw [line width=1pt] (-2.7801471275980205,-2.898564405095965)--  (-2.7801471275980205,-2.898564405095965)-- (-2.774353840003287,-2.915944267880166)-- (-2.768560552408553,-2.9333241306643676)-- (-2.7627672648138195,-2.950703993448569)-- (-2.7511806896243516,-2.96808385623277)-- (-2.739594114434884,-2.9854637190169715)-- (-2.7338008268401506,-3.002843581801173)-- (-2.722214251650683,-3.020223444585374)-- (-2.7048343888664816,-3.0376033073695754)-- (-2.6874545260822806,-3.054983170153777)-- (-2.670074663298079,-3.0665697453432443)-- (-2.646901512919144,-3.078156320532712)-- (-2.6295216501349428,-3.0897428957221793)-- (-2.6121417873507418,-3.0955361833169133)-- (-2.5947619245665403,-3.101329470911647)-- (-2.5773820617823393,-3.1129160461011143)-- (-2.560002198998138,-3.1187093336958482)-- (-2.536829048619203,-3.1245026212905818)-- (-2.5194491858350014,-3.1302959088853157)-- (-2.5020693230508004,-3.1360891964800492)-- (-2.4788961726718655,-3.1360891964800492)-- (-2.4557230222929305,-3.1360891964800492)-- (-2.438343159508729,-3.141882484074783)-- (-2.415170009129794,-3.141882484074783)-- (-2.391996858750859,-3.141882484074783)-- (-2.368823708371924,-3.141882484074783);\draw [line width=1pt] (2.2347381775266486,-1.5903666558252543)--  (2.2347381775266486,-1.5903666558252543)-- (2.218938302268284,-1.5745667805668895)-- (2.2084050520960408,-1.5587669053085247)-- (2.203138427009919,-1.54296703005016)-- (2.1978718019237977,-1.5271671547917949)-- (2.192605176837676,-1.51136727953343)-- (2.182071926665433,-1.4903007791889438)-- (2.176805301579311,-1.474500903930579)-- (2.1715386764931894,-1.4587010286722142)-- (2.166272051407068,-1.4376345283277276)-- (2.166272051407068,-1.4165680279832413)-- (2.1610054263209464,-1.4007681527248765)-- (2.1610054263209464,-1.37970165238039)-- (2.1610054263209464,-1.3586351520359037)-- (2.1610054263209464,-1.3375686516914171)-- (2.1610054263209464,-1.3165021513469308)-- (2.1610054263209464,-1.2954356510024443)-- (2.1610054263209464,-1.274369150657958)-- (2.166272051407068,-1.2585692753995932)-- (2.176805301579311,-1.2427694001412284)-- (2.1873385517515542,-1.2269695248828636)-- (2.203138427009919,-1.2111696496244986)-- (2.218938302268284,-1.1953697743661338)-- (2.2347381775266486,-1.179569899107769)-- (2.2505380527850134,-1.1637700238494042)-- (2.266337928043378,-1.1585033987632827)-- (2.2768711782156217,-1.1427035235049179)-- (2.2926710534739865,-1.1321702733326746)-- (2.3084709287323513,-1.126903648246553)-- (2.324270803990716,-1.1216370231604313)-- (2.340070679249081,-1.1163703980743098)-- (2.355870554507445,-1.111103772988188)--  (2.282137803301743,-1.1427035235049179)--  (2.282137803301743,-1.1427035235049179)-- (2.2979376785601078,-1.1585033987632827)-- (2.3084709287323513,-1.1743032740216475)-- (2.3137375538184726,-1.1901031492800123)-- (2.324270803990716,-1.205903024538377)-- (2.334804054162959,-1.2217028997967418)-- (2.340070679249081,-1.2375027750551066)-- (2.345337304335202,-1.2533026503134714)-- (2.350603929421324,-1.2691025255718362)-- (2.355870554507445,-1.2849024008302012)-- (2.361137179593567,-1.300702276088566)-- (2.361137179593567,-1.3217687764330524)-- (2.361137179593567,-1.3428352767775389)-- (2.361137179593567,-1.3639017771220252)-- (2.355870554507445,-1.37970165238039)-- (2.350603929421324,-1.3955015276387548)-- (2.340070679249081,-1.4113014028971196)-- (2.3295374290768374,-1.4271012781554844)-- (2.3190041789045943,-1.4429011534138494)-- (2.3032043036462295,-1.4534344035860924)-- (2.2874044283878647,-1.4692342788444572)-- (2.2716045531295,-1.4797675290167005)-- (2.255804677871135,-1.4903007791889438)-- (2.2400048026127704,-1.500834029361187)-- (2.2242049273544056,-1.51136727953343)-- (2.2084050520960408,-1.5219005297056734)-- (2.192605176837676,-1.5324337798779166);\draw [line width=1pt] (2.440136555885391,-0.7161068915290678)--  (2.440136555885391,-0.7161068915290678)-- (2.440136555885391,-0.6897737660984598)-- (2.434869930799269,-0.673973890840095)-- (2.429603305713148,-0.6529073904956085)-- (2.424336680627026,-0.6318408901511221)-- (2.413803430454783,-0.6107743898066357)-- (2.4085368053686613,-0.5897078894621492)-- (2.4085368053686613,-0.5686413891176628)-- (2.4032701802825396,-0.552841513859298)-- (2.4032701802825396,-0.5317750135148116)-- (2.4032701802825396,-0.5107085131703252)-- (2.4032701802825396,-0.4896420128258388)-- (2.4032701802825396,-0.4685755124813524)-- (2.4085368053686613,-0.4527756372229875)-- (2.413803430454783,-0.43697576196462273)-- (2.424336680627026,-0.4211758867062579)-- (2.429603305713148,-0.4053760114478931)-- (2.440136555885391,-0.3895761361895283)-- (2.4559364311437557,-0.37377626093116345)-- (2.4664696813159988,-0.35797638567279866)-- (2.4822695565743635,-0.34744313550055544)-- (2.492802806746607,-0.3316432602421906)-- (2.508602682004972,-0.32111001006994744)-- (2.519135932177215,-0.3053101348115826)-- (2.5349358074355797,-0.2947768846393394)-- (2.5454690576078227,-0.2789770093809746)-- (2.5612689328661875,-0.27371038429485295)-- (2.5770688081245523,-0.25791050903648816)-- (2.592868683382917,-0.2526438839503665)-- (2.608668558641282,-0.24211063377812334)-- (2.6244684338996467,-0.23684400869200173)-- (2.6402683091580115,-0.23157738360588012)-- (2.661334809502498,-0.22631075851975851)-- (2.677134684760863,-0.2210441334336369)-- (2.6929345600192276,-0.2157775083475153)-- (2.7140010603637137,-0.2157775083475153)-- (2.7350675607082,-0.2157775083475153)-- (2.7561340610526863,-0.2210441334336369)-- (2.771933936311051,-0.22631075851975851);\draw [line width=1pt] (2.608668558641282,-0.24737725886424494)--  (2.608668558641282,-0.24737725886424494)-- (2.608668558641282,-0.26844375920873137)-- (2.608668558641282,-0.2895102595532178)-- (2.608668558641282,-0.3105767598977042)-- (2.6139351837274036,-0.326376635156069)-- (2.6139351837274036,-0.34744313550055544)-- (2.6139351837274036,-0.37377626093116345)-- (2.6139351837274036,-0.3948427612756499)-- (2.608668558641282,-0.4159092616201363)-- (2.60340193355516,-0.4317091368785011)-- (2.598135308469039,-0.44750901213686595)-- (2.5876020582967953,-0.46330888739523074)-- (2.5718021830384306,-0.4791087626535956)-- (2.556002307780066,-0.4949086379119604)-- (2.5454690576078227,-0.5107085131703252)-- (2.529669182349458,-0.5159751382564468)-- (2.513869307091093,-0.5317750135148116)-- (2.4980694318327283,-0.5423082636870549)-- (2.4822695565743635,-0.552841513859298)-- (2.4664696813159988,-0.5581081389454197)-- (2.450669806057634,-0.5686413891176628)-- (2.429603305713148,-0.5686413891176628)-- (2.413803430454783,-0.5739080142037845)-- (2.3927369301102965,-0.5739080142037845);\draw [line width=1pt,dash pattern=on 1pt off 1pt,color=qqqqqq] (3.1724607741105935,1.22754109650411)--  (3.1724607741105935,1.22754109650411)-- (3.1724607741105935,1.204367946125175)-- (3.1724607741105935,1.1811947957462399)-- (3.1724607741105935,1.158021645367305)-- (3.1724607741105935,1.1348484949883697)-- (3.16666748651586,1.1174686322041685)-- (3.16666748651586,1.0942954818252333)-- (3.160874198921126,1.076915619041032)-- (3.1550809113263925,1.0595357562568306)-- (3.1550809113263925,1.0363626058778956)-- (3.1492876237316585,1.0189827430936942)-- (3.143494336136925,0.9958095927147592)-- (3.143494336136925,0.972636442335824)-- (3.137701048542191,0.9552565795516228)-- (3.137701048542191,0.9320834291726876)-- (3.1319077609474575,0.9147035663884864)-- (3.1319077609474575,0.8915304160095512)-- (3.1319077609474575,0.8683572656306161)-- (3.1261144733527235,0.8509774028464148)-- (3.12032118575799,0.8335975400622135)-- (3.114527898163256,0.8162176772780122)-- (3.1087346105685225,0.7988378144938109)-- (3.1029413229737886,0.7814579517096095)-- (3.097148035379055,0.7640780889254082)-- (3.091354747784321,0.7466982261412068)-- (3.0855614601895875,0.7293183633570056)-- (3.0797681725948536,0.7119385005728043)-- (3.07397488500012,0.6887653501938692)-- (3.068181597405386,0.665592199814934)-- (3.0623883098106526,0.6482123370307328);\draw [line width=1pt,dash pattern=on 1pt off 1pt,color=qqqqqq] (3.068181597405386,0.7003519253833367)--  (3.068181597405386,0.7003519253833367)-- (3.068181597405386,0.6713854874096679)-- (3.07397488500012,0.6540056246254665)-- (3.0797681725948536,0.6366257618412652)-- (3.0855614601895875,0.6192458990570638)-- (3.091354747784321,0.5960727486781288)-- (3.097148035379055,0.5786928858939274)-- (3.1029413229737886,0.5613130231097261)-- (3.1029413229737886,0.538139872730791)-- (3.1087346105685225,0.5207600099465897)-- (3.114527898163256,0.5033801471623884)-- (3.12032118575799,0.4860002843781871)-- (3.1261144733527235,0.4512405588097844)-- (3.1319077609474575,0.4338606960255831)-- (3.137701048542191,0.4164808332413818)-- (3.143494336136925,0.39910097045718046)-- (3.1492876237316585,0.3817211076729791)-- (3.1550809113263925,0.36434124488877784)-- (3.1550809113263925,0.3411680945098427)-- (3.1550809113263925,0.3179949441309076);\draw [line width=1pt,dash pattern=on 1pt off 1pt,color=qqqqqq] (3.16666748651586,0.28902850615723874)--  (3.16666748651586,0.28902850615723874)-- (3.189840636894795,0.28902850615723874)-- (3.21301378727373,0.28902850615723874)-- (3.2477735128421323,0.28902850615723874)-- (3.2709466632210673,0.283235218562505)-- (3.299913101194736,0.283235218562505)-- (3.3172929639789372,0.2774419309677712)-- (3.3346728267631383,0.26585535577830366)-- (3.3520526895473397,0.2600620681835699)-- (3.3694325523315407,0.2542687805888361)-- (3.386812415115742,0.24268220539936858)-- (3.404192277899943,0.22530234261516724)-- (3.4215721406841446,0.2137157674256997)-- (3.4389520034683456,0.1963359046414984)-- (3.456331866252547,0.19054261704676462)-- (3.473711729036748,0.17316275426256328)-- (3.4910915918209495,0.16157617907309574)-- (3.5084714546051505,0.14419631628889443)-- (3.520058029794618,0.1268164535046931)-- (3.5374378925788195,0.10943659072049178)-- (3.543231180173553,0.09205672793629045)-- (3.5606110429577544,0.08626344034155668);\draw [line width=1pt,dotted,color=qqqqqq] (3.160874198921126,1.2507142468830452)--  (3.160874198921126,1.2507142468830452)-- (3.160874198921126,1.22754109650411)-- (3.16666748651586,1.2101612337199088)-- (3.1782540617053274,1.1927813709357076)-- (3.184047349300061,1.1754015081515061)-- (3.1956339244895284,1.1522283577725712)-- (3.2014272120842624,1.1348484949883697)-- (3.207220499678996,1.1116753446094347)-- (3.21301378727373,1.0942954818252333)-- (3.2188070748684634,1.0711223314462983)-- (3.230393650057931,1.0537424686620969)-- (3.236186937652665,1.0363626058778956)-- (3.2419802252473984,1.0131894554989604)-- (3.253566800436866,0.9958095927147592)-- (3.2593600880316,0.9784297299305579)-- (3.2651533756263333,0.9552565795516228)-- (3.276739950815801,0.9320834291726876)-- (3.282533238410535,0.9089102787937525)-- (3.2883265260052683,0.8915304160095512)-- (3.2941198136000023,0.87415055322535)-- (3.3057063887894698,0.8567706904411486)-- (3.3114996763842033,0.8393908276569473)-- (3.3172929639789372,0.8220109648727459)-- (3.3230862515736708,0.8046311020885446)-- (3.3346728267631383,0.7872512393043433)-- (3.3462594019526057,0.769871376520142)-- (3.3520526895473397,0.7524915137359407)-- (3.3578459771420732,0.7351116509517394)-- (3.363639264736807,0.717731788167538)-- (3.3694325523315407,0.7003519253833367)-- (3.381019127521008,0.6829720625991353)-- (3.386812415115742,0.665592199814934)-- (3.3926057027104757,0.6424190494359989)-- (3.404192277899943,0.6192458990570638)-- (3.409985565494677,0.6018660362728625)-- (3.4157788530894106,0.5844861734886613)-- (3.4215721406841446,0.5555197355149923)-- (3.427365428278878,0.538139872730791)-- (3.427365428278878,0.5149667223518559)-- (3.4389520034683456,0.49758685956765464)-- (3.4447452910630796,0.4744137091887195)-- (3.450538578657813,0.4512405588097844)-- (3.456331866252547,0.4338606960255831)-- (3.4621251538472806,0.4164808332413818)-- (3.4621251538472806,0.3933076828624467)-- (3.4679184414420146,0.3701345324835116)-- (3.473711729036748,0.3411680945098427)-- (3.473711729036748,0.3122016565361739)-- (3.479505016631482,0.283235218562505)-- (3.4852983042262156,0.26585535577830366)-- (3.4910915918209495,0.24268220539936858)-- (3.496884879415683,0.22530234261516724)-- (3.502678167010417,0.20792247983096593)-- (3.5084714546051505,0.19054261704676462)-- (3.520058029794618,0.17316275426256328)-- (3.5316446049840855,0.15578289147836197)-- (3.543231180173553,0.13840302869416066)-- (3.5548177553630205,0.12102316590995932);\draw [line width=1pt,color=qqqqqq] (0.2357545232761704,2.173815026752726)--  (0.2357545232761704,2.173815026752726)-- (0.26120038037851945,2.1822969791201756)-- (0.2951281898483182,2.199260883855075)-- (0.3205740469506672,2.2162247885899746)-- (0.34601990405301625,2.2247067409574246)-- (0.3714657611553653,2.241670645692324)-- (0.4138755229926137,2.2501525980597736)-- (0.4478033324624124,2.2586345504272236)-- (0.4902130942996608,2.2586345504272236)-- (0.5326228561369092,2.2586345504272236)-- (0.5750326179741576,2.2586345504272236)-- (0.6089604274439563,2.2586345504272236)-- (0.6428882369137551,2.2586345504272236)-- (0.6683340940161041,2.2501525980597736)-- (0.6937799511184531,2.241670645692324)-- (0.7192258082208022,2.2247067409574246)-- (0.7446716653231512,2.207742836222525)-- (0.7616355700580505,2.1822969791201756)-- (0.7785994747929499,2.1568511220178266)-- (0.7870814271603995,2.1314052649154775)-- (0.7955633795278493,2.105959407813128)-- (0.8040453318952989,2.080513550710779)-- (0.8125272842627486,2.05506769360843)-- (0.829491188997648,2.0296218365060805)-- (0.8379731413650976,2.0041759794037315)-- (0.8464550937325473,1.9787301223013822)-- (0.854937046099997,1.953284265199033)-- (0.854937046099997,1.919356455729234)-- (0.854937046099997,1.8854286462594352)-- (0.8634189984674467,1.859982789157086)-- (0.8634189984674467,1.826054979687287)-- (0.8634189984674467,1.792127170217488)-- (0.8634189984674467,1.7581993607476891)-- (0.8634189984674467,1.7242715512778901);\draw [line width=1pt,color=qqqqqq] (0.3290559993181169,2.233188693324874)--  (0.3290559993181169,2.233188693324874)-- (0.3290559993181169,2.199260883855075)-- (0.3290559993181169,2.1653330743852766)-- (0.3290559993181169,2.1314052649154775)-- (0.34601990405301625,2.105959407813128)-- (0.3629838087879156,2.072031598343329)-- (0.379947713522815,2.04658574124098)-- (0.405393570625164,2.021139884138631)-- (0.42235747536006335,1.9956940270362817)-- (0.43932138009496274,1.9702481699339325)-- (0.4478033324624124,1.9448023128315832)-- (0.4732491895647614,1.9278384080966837)-- (0.49869504666711045,1.9023925509943345)-- (0.5326228561369092,1.893910598626885)-- (0.5580687132392582,1.8769466938919854)-- (0.5835145703416073,1.8684647415245357)-- (0.6089604274439563,1.859982789157086)-- (0.6428882369137551,1.859982789157086)-- (0.6768160463835537,1.859982789157086)-- (0.7107438558533524,1.859982789157086)-- (0.7446716653231512,1.859982789157086)-- (0.7701175224255002,1.8684647415245357)-- (0.7955633795278493,1.8769466938919854)-- (0.8210092366301983,1.8854286462594352)-- (0.8464550937325473,1.893910598626885)-- (0.8719009508348964,1.9023925509943345);\draw [line width=1pt,color=qqqqqq] (1.4317098070865752,1.5291866468265463)--  (1.4317098070865752,1.5291866468265463)-- (1.4571556641889243,1.537668599193996)-- (1.4826015212912733,1.5546325039288955)-- (1.5080473783936224,1.5631144562963453)-- (1.5334932354959714,1.571596408663795)-- (1.5589390925983204,1.5800783610312445)-- (1.5928669020681192,1.5800783610312445)-- (1.6183127591704682,1.5885603133986943)-- (1.6522405686402668,1.5885603133986943)-- (1.6861683781100656,1.5800783610312445)-- (1.7116142352124146,1.571596408663795)-- (1.7370600923147637,1.5631144562963453)-- (1.7625059494171127,1.5546325039288955)-- (1.7879518065194617,1.537668599193996)-- (1.8133976636218108,1.5291866468265463)-- (1.8388435207241598,1.503740789724197)-- (1.8642893778265088,1.4867768849892975)-- (1.8897352349288579,1.469812980254398)-- (1.9066991396637571,1.444367123152049)-- (1.9236630443986567,1.4189212660496997)-- (1.940626949133556,1.3934754089473504)-- (1.9575908538684552,1.3680295518450012)-- (1.966072806235905,1.3425836947426522)-- (1.9745547586033547,1.317137837640303)-- (1.9830367109708043,1.2916919805379536)-- (1.991518663338254,1.2662461234356044)-- (1.991518663338254,1.2323183139658056)-- (2.0000006157057038,1.2068724568634563)-- (2.0000006157057038,1.1729446473936573);\draw [line width=1pt,color=qqqqqq] (1.5674210449657702,1.5800783610312445)--  (1.5674210449657702,1.5800783610312445)-- (1.5674210449657702,1.5461505515614458)-- (1.5759029973332197,1.5207046944590965)-- (1.5843849497006695,1.4952588373567473)-- (1.5928669020681192,1.469812980254398)-- (1.6013488544355687,1.444367123152049)-- (1.6098308068030185,1.4189212660496997)-- (1.6352766639053675,1.3934754089473504)-- (1.6522405686402668,1.3680295518450012)-- (1.6776864257426158,1.3595475994775514)-- (1.7031322828449649,1.3425836947426522)-- (1.728578139947314,1.3256197900077527)-- (1.754023997049663,1.317137837640303)-- (1.779469854152012,1.3086558852728531)-- (1.804915711254361,1.3001739329054034)-- (1.83036156835671,1.2916919805379536)-- (1.855807425459059,1.283210028170504)-- (1.8897352349288579,1.283210028170504)-- (1.9236630443986567,1.283210028170504)-- (1.9575908538684552,1.283210028170504)-- (1.991518663338254,1.283210028170504)-- (2.0084825680731533,1.3086558852728531);\draw [line width=1pt,color=qqqqqq] (-2.9073951095685193,-2.8909205963069042)--  (-2.9073951095685193,-2.8909205963069042)-- (-2.875532027797484,-2.8909205963069042)-- (-2.8500415623806554,-2.8909205963069042)-- (-2.824551096963827,-2.8909205963069042)-- (-2.7990606315469986,-2.8909205963069042)-- (-2.77357016613017,-2.8909205963069042)-- (-2.754452317067549,-2.87817536359849)-- (-2.735334468004927,-2.8718027472442826)-- (-2.7034713862338915,-2.8718027472442826)-- (-2.6843535371712703,-2.8654301308900756)-- (-2.6588630717544417,-2.8654301308900756)-- (-2.633372606337613,-2.8654301308900756)-- (-2.607882140920785,-2.8654301308900756)-- (-2.582391675503956,-2.8654301308900756)-- (-2.5569012100871276,-2.8654301308900756)-- (-2.5377833610245064,-2.8718027472442826)-- (-2.5186655119618853,-2.87817536359849)-- (-2.4995476628992637,-2.884547979952697)-- (-2.4804298138366425,-2.8909205963069042)-- (-2.461311964774021,-2.8972932126611113)-- (-2.4421941157113998,-2.910038445369526)-- (-2.423076266648778,-2.92278367807794)-- (-2.410331033940364,-2.9419015271405615)-- (-2.3975858012319495,-2.961019376203183)-- (-2.3912131848777425,-2.9801372252658043)-- (-2.3784679521693284,-2.999255074328426)-- (-2.3720953358151213,-3.0183729233910475)-- (-2.3657227194609143,-3.0374907724536686)-- (-2.359350103106707,-3.05660862151629)-- (-2.3466048703982927,-3.075726470578912)-- (-2.3402322540440856,-3.094844319641533)-- (-2.3338596376898786,-3.1139621687041545)-- (-2.3338596376898786,-3.1458252504751902)-- (-2.3274870213356715,-3.164943099537812)-- (-2.3274870213356715,-3.1904335649546405)-- (-2.3338596376898786,-3.215924030371469)-- (-2.3338596376898786,-3.2414144957882978);\draw [line width=1pt,color=qqqqqq] (-3.4618127323845393,-2.342875589845089)--  (-3.4618127323845393,-2.342875589845089)-- (-3.4363222669677107,-2.342875589845089)-- (-3.410831801550882,-2.342875589845089)-- (-3.385341336134054,-2.342875589845089)-- (-3.359850870717225,-2.342875589845089)-- (-3.3343604053003966,-2.342875589845089)-- (-3.3088699398835684,-2.342875589845089)-- (-3.2833794744667397,-2.342875589845089)-- (-3.257889009049911,-2.349248206199296)-- (-3.232398543633083,-2.3556208225535036)-- (-3.2132806945704613,-2.3619934389077106)-- (-3.19416284550784,-2.3747386716161247)-- (-3.1750449964452185,-2.381111287970332)-- (-3.1559271473825974,-2.3938565206787463)-- (-3.136809298319976,-2.4002291370329534)-- (-3.1240640656115617,-2.419346986095575)-- (-3.1113188329031476,-2.4384648351581966)-- (-3.10494621654894,-2.4575826842208177)-- (-3.092200983840526,-2.4767005332834393)-- (-3.079455751132112,-2.495818382346061)-- (-3.073083134777905,-2.514936231408682)-- (-3.0667105184236974,-2.5340540804713036)-- (-3.0603379020694903,-2.5531719295339252)-- (-3.0603379020694903,-2.578662394950754)-- (-3.0603379020694903,-2.6041528603675825)-- (-3.0603379020694903,-2.629643325784411)-- (-3.0603379020694903,-2.6551337912012394)-- (-3.0603379020694903,-2.680624256618068)-- (-3.0603379020694903,-2.7061147220348967)-- (-3.0667105184236974,-2.7252325710975183);\draw [line width=1pt,color=qqqqqq] (-3.3725961034256393,-2.3556208225535036)--  (-3.3725961034256393,-2.3556208225535036)-- (-3.3725961034256393,-2.381111287970332)-- (-3.3725961034256393,-2.412974369741368)-- (-3.3662234870714323,-2.432092218803989)-- (-3.359850870717225,-2.4512100678666107)-- (-3.347105638008811,-2.4703279169292323)-- (-3.3343604053003966,-2.4894457659918534)-- (-3.3216151725919825,-2.508563615054475)-- (-3.3088699398835684,-2.5276814641170966)-- (-3.296124707175154,-2.546799313179718)-- (-3.2770068581125327,-2.5531719295339252)-- (-3.257889009049911,-2.5659171622423393)-- (-3.23877115998729,-2.578662394950754)-- (-3.2196533109246683,-2.585035011304961)-- (-3.200535461862047,-2.591407627659168)-- (-3.1814176127994256,-2.597780244013375)-- (-3.1622997637368044,-2.6041528603675825)-- (-3.143181914674183,-2.6105254767217896)-- (-3.1240640656115617,-2.6168980930759966)-- (-3.098573600194733,-2.6168980930759966)-- (-3.073083134777905,-2.6232707094302037)-- (-3.047592669361076,-2.6232707094302037)-- (-3.0284748202984546,-2.629643325784411);\begin{scriptsize}\draw [fill=black] (-5.195370234049666,1.375370234049646) circle (0.5pt);\draw [fill=black] (-4.518674516012442,2.4062248386708074) circle (0.5pt);\draw [fill=black] (-5,2.24) circle (0.5pt);\draw [fill=black] (-5.16,2.02) circle (0.5pt);\draw [fill=black] (-4.3923877529478395,1.9288317365306689) circle (0.5pt);\end{scriptsize}
\draw (-1,0) node[above] {$Convex\ core$};
\draw (3.5,2) node[above] {$Rank\ one\ cusp$};
\draw [->] (3.5,2) -- (3.5,1.25);
\draw (0,-2.5) node[below] {$Boundary$};
\draw [->] (1.25,-2.75) -- (2,-2.5);
\draw [->] (-1,-3) -- (-2,-3.5);
\draw (-5,0) node[above] {$Maximal$};
\draw (-5,-0.5) node[above] {$rank\ cusp$};
\draw [->] (-5,0.5) -- (-5,1);
\end{tikzpicture}
	\end{center}
	\caption{Convex core}
\end{figure}
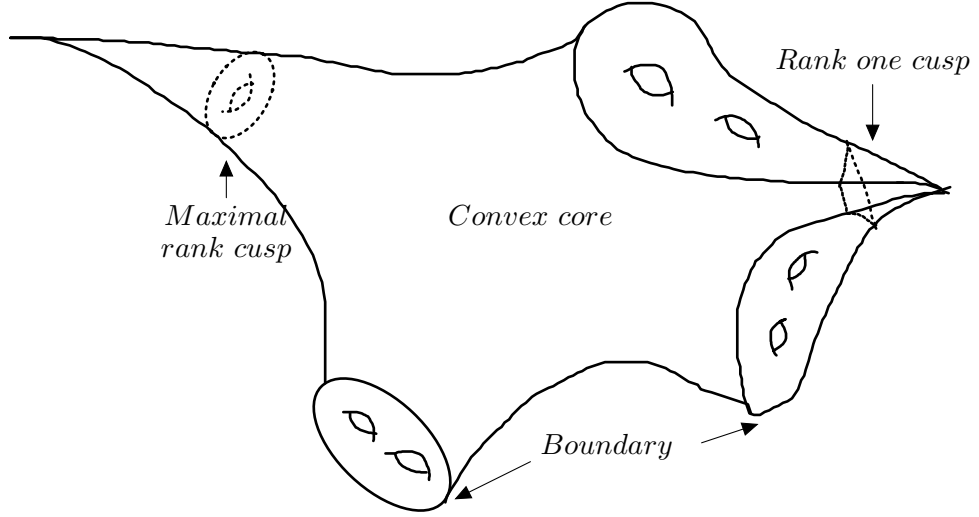

Our covering has the advantage that $M_3$ is compact, and $M_1,M_2$ have quasi-warped product Riemannian structure.
\begin{prop}\label{prop:convexcorecompact}
	The set $M_3$ is relatively compact in $M$.
\end{prop}
\begin{proof}
	By Definition \ref{def:geofinite}, we have that $M_c-\mathring{\cusp'}$ is compact in $M_c$. Therefore $C_2(M)-\mathring{\cusp'}$ is compact in $M_c$. It is also contained in $M$, hence it is a compact subset of $M$.
\end{proof}
\subsection{The energy form and projections}
\label{sec:enefor}
We keep the assumptions on $M$ of Sec.\ref{sec:goopar}, that $M$ is a geometrically finite locally symmetric manifold of real rank one with at most one end. Recall that $E$ is the energy form defined in \eqref{equ:energy}.
\begin{lem}\label{lem:enevar}
	Let $f\in C^{\infty}_c(M)$, and let $\varphi$ be a smooth function. We have
	$$E(\varphi f)=\int_{M}\varphi^2(\|\nabla f\|^2-\ell^2f^2)\dd vol-\int_{M}f^2\varphi\Delta\varphi\dd vol.$$
\end{lem}
This is a direct consequence of Lemma \ref{lem:barta}. The following proposition says that in order to calculate the energy form on the entire manifold, it is sufficient to calculate the energy forms on $M_1,M_2,M_3$ and an error term. We want to separate the energy formula, but we need a partition of unity such that the square root of the partition function is also smooth. The exact choice of the partition function is not important and we take $\theone=\frac{\pi}{2}\bar{\varphi_1},\thetwo=\frac{\pi}{2}\bar{\varphi_2}$ and $\varphi_1=\sin\theone,\ \varphi_2=\sin\thetwo$. 
% how to write these partition functions, 
\begin{prop} For $f\in C_c^\infty(M)$, we have
	\begin{equation}\label{equ:enedec}
	\begin{split}
	E(f)&=\int_{M_3}\left((1-\varphi_1^2-\varphi_2^2)(\|\nabla f\|^2-\ell^2f^2)+(\varphi_1\Delta\varphi_1+\varphi_2\Delta\varphi_2)f^2\right)\dd vol\\ 
	&+E(\varphi_1 f)+E(\varphi_2 f)-\int_{M_3^c}\|\nabla \theone\|^2f^2\dd vol.
	\end{split}
	\end{equation}
\end{prop}
\begin{proof} By Lemma \ref{lem:enevar}, we have
	$$E(\varphi_2 f)+E(\varphi_1 f)=\int (\varphi_2^2+\varphi_1^2)(\|\nabla f\|^2-\ell^2f^2)-\int (\varphi_2\Delta\varphi_2+\varphi_1\Delta\varphi_1)f^2.
	$$
	Since $\theone+\thetwo=\pi/2$ outside of $M_3$, we have for $x\in M_3^c$
	\begin{align*}
	\varphi_2\Delta\varphi_2+\varphi_1\Delta\varphi_1&=
	\sin\theone\Delta\sin\theone+\sin\thetwo\Delta\sin\thetwo\\
	&=\frac{1}{2}(\Delta(\sin^2\theone+\cos^2\theone)-\|\nabla\sin\theone\|^2-\|\nabla\cos\theone\|^2)=-\|\nabla\theone\|^2.
	\end{align*}
	The proof is complete.
\end{proof}
We want to prove that $E(\varphi_1 f)$ and $E(\varphi_2 f)$ are positive after adding an integral over a compact subset, and to give an estimate of the error term $\int_{M_3^c}\|\nabla \theta\|^2f^2$. For this purpose, we will study the following projections. 

Let $\tilde{\rho}$ be the nearest point retraction from $X$ to $W'$. We can extend this map continuously to $X\cup \Omega$, such that if $\xi$ is in $\Omega$, then $\tilde{\rho}(\xi)$ is the first point of contact of $W'$ with an expanding family of horoballs based at $\xi$. (See \cite[Lemma 2.2.4]{bowditch1995} and the explication before Lemma 2.2.4 for more details.) This retraction descends to a map
$$\rho:M_c\rightarrow \convex,$$ which is also continuous by the openness of the covering map.

In the cusp region, recall that we will use the half space model $(\Pi\backslash N)\times \Rplus$ for $M_c(\Pi)$. We write $M_c(\Pi)=(\Pi\backslash N)\times \bb R_{\geq 0}$, which means the equality holds under the coordinate map. Others equalities are similar.
By definition, $\cusp'$ is isometric to a subset of $M_c(\Pi)$, where $\Pi$ is a discrete parabolic subgroup. Hence we can identify the two sets.  By Definition \ref{defi:end} and Proposition \ref{prop:end}, the complement of the cusp region $\cusp'$ in $M_c(\Pi)$ is relatively compact. Hence we can suppose that $(\cusp')^c$ is contained in $(\Pi\backslash N)\times[0,1]$. (Otherwise, we can change the coordinate.) Let $H$ be the quotient of a horoball based at $\infty$, defined by $H=(\Pi\backslash N)\times [1,\infty)\subset \cusp'$. We define $proj_H$ to be the nearest point retraction from $M_c(\Pi)$ to $H$. For a point $(\eta,y)\in M_c(\Pi)$ with $y<1$, the map is given by $proj_H(\eta,y)=(\eta,1)$.
\begin{lem}\label{lem:proj}
	Let $K=\rho(M_2)$ and let $L=proj_{H}(\cusp-\cusp')$. Then $K,L$ are relatively compact in $M$. 
	
	When $\cusp$ is of maximal rank, we can take $\cusp=(\Pi\backslash N)\times \R_{\geq c_1}$ and $\cusp'=(\Pi\backslash N)\times\R_{\geq c_2}$ for some $0<c_1<c_2\leq 1$ and we have $L= (\Pi\backslash N)\times \{ 1\}$.
\end{lem} 
\begin{proof}
	Since $M_c$ has only one end, by Definition \ref{def:geofinite}, the set $M_2=D^c-\cusp'$ and $\cusp-\cusp'$ are relatively compact in $M_c$.  The continuity of $\rho$ and $proj_H$ implies that $\rho(M_2)$ and $proj_{H}(\cusp-\cusp')$ are relatively compact in $\convex$ and $H$. The latter two sets are in $M$, hence we have the result. 
	
	The last assertion is due to Proposition \ref{prop:mrank}.
\end{proof}
\begin{rem}
This is an important consequence of geometrical finiteness, and the key difficulty is hid in the definition and the existence of the standard cusp region. 
\end{rem}
For $x$ in $M$,  let 
\begin{equation}\label{equ:t1t2}
t_1(x)=d(x,H), t_2(x)=d(x,\convex).
\end{equation}
Later we will see that $t_1,t_2$ are the geometric descriptions of the coordinate $t$ in cusps and the complement of convex set. 
\subsection{Positivity of the energy form}
\label{sec:positive}
In this part, we take into account the topology of the whole manifold, together with the results in standard cusp regions and the complement of convex subsets, to prove the positivity. 
\begin{prop}[the Lax-Phillips inequality]\label{prop:gradient}
	Let $M$ be a geometrically finite locally symmetric manifold of real rank one. There exist a relatively compact open set $U$ in $M$ with smooth boundary and a constant $C_U>0$ such that the following holds. For any compact set $V$ in $M$ there exists $\epsilon_V>0$ such that for all complex valued function $f\in C_c^\infty(M)$ we have
	\begin{equation}\label{ineq:gradient}
		E(f)+C_U\int_{U}|f|^2\dd vol\geq \frac{1}{4}\int_U \|\nabla f\|^2\dd vol+\epsilon_V\int_V|f|^2\dd vol.
	\end{equation}
\end{prop}
\begin{proof}[Proof of Theorem \ref{the:main}]
	Our main theorem of this manuscript follows from Proposition \ref{prop:gradient} and \ref{prop:main} with $\eigen=\ell^2$. 
\end{proof}
It remains to prove Proposition \ref{prop:gradient}.
\begin{rem}
	For the simplicity of the exposition, here we will only prove the case that $M$ has only maximal rank cusps ($M$ may have no cusp). For the general case, please see Section \ref{sec:app}. The idea of the proof is the same, but the appearance of non-maximal cusps will add some technical difficulties.
	
	It is sufficient to prove this inequality for real valued functions. The complex version is immediate by separating $f=f_1+if_2$ with $f_1,f_2$ real valued and using the real version for each component $f_1,f_2$.
\end{rem}
Recall that $M_1,M_2$ are subsets of $M_c$, which may intersect the ideal boundary $M_I$. Recall that $t_1,t_2$ are functions on $M$ defined in \eqref{equ:t1t2}.
\begin{lem}\label{lem:cusp} 
	In the standard cusp region, there exist a compact set $U_1$ in $M_1\cap M$ and a constant $C_1>0$ such that the following holds. For any compact set $V$ in $M_1\cap M$ there exists $\epsilon_V>0$ such that for all $f\in C_c^{\infty}(M_1\cap M)$ we have
	\begin{equation}\label{ineq:cusp} 
	\begin{split}
	E(f)+C_1\int_{U_1}f^2\dd vol\geq \frac{1}{2}\int\|\nabla (e^{\ell t_1}f)\|^2e^{-2\ell t_1}\dd vol+\epsilon_V\int_{V} f^2\dd vol.
	\end{split}
	\end{equation}
\end{lem}
\begin{proof}
Recall that we identify $M_1=\cusp'$ as a subset of $(\Pi\backslash N)\times\bb R_{\geq 0}$. Due to the choice of $H$, we have $\cusp-\cusp'\subset H^c$, and the coordinate $y$ of a point $x$ under the horospherical model satisfies $y=e^{-t_1(x)}$ for $x\in H^c$. We have that $t_1(x)$ equals the coordinate $t$ in horospherical model when $x\in H^c$. Using Lemma \ref{prop:cusp} with compact set $L$ and constant $C=2$, and adding \eqref{ineq:cusp0} and \eqref{ineq:cusp1}, we obtain inequality \eqref{ineq:cusp} with the last term replaced by $\int_{L\times \bb R_{\geq 0}}e^{-t_1}f^2\dd vol $. Since $\cusp$ is of maximal rank, by Lemma \ref{lem:proj} we have $L\times \bb R_{\geq 0}\simeq M_c(\Pi)\supset M_1$. The proof is complete.
\end{proof}
\begin{lem}\label{lem:complement}
 In the complement of the convex core, there exist a compact set $U_2$ in $M_2\cap M$ and a constant $C_2>0$ such that the following holds. For any compact set $V$ in $M_2\cap M$ there exists $\epsilon_V>0$ such that for all $f\in C_c^{\infty}(M_2\cap M)$ we have
\begin{align}
 E(f)+C_2\int_{U_2}f^2 \dd vol\geq &\frac{1}{2}\int\|\nabla (e^{\ell t_2}f)\|^2e^{-2\ell t_2}\dd vol+\epsilon_V\int_Vf^2 \dd vol. \label{ineq:complement}
\end{align}
\end{lem}
\begin{proof}
	We first verify that $\convex$ satisfies the conditions in Lemma \ref{volumefunction}.
	
	Recall that $\convex=\Gamma\backslash W'$, where $W'$ is a convex subset of $X$ with smooth boundary. By convexity, the normal exponential map, given by $\exp_x(t\normal(x))$, is a diffeomorphism from $\partial W'\times \bb R_{\geq 0}$ onto $X-\mathring{W'}$, and satisfies $t=d(\exp_x(t\normal(x)),W')$. With the help the nearest point retraction $\tilde{\rho}$, the inverse of the normal exponential map from $X-\mathring{W'}$ to $\partial W'\times \bb R_{\geq 0}$ is given by 
	\[x\mapsto (\tilde{\rho}(x), d(x,W')). \] 
	Descend to the quotient space. Let $S$ be the boundary of $\convex$.  Then the normal exponential map $\varbet:(x,t)\mapsto\exp_x(t\normal(x))$ from $S\times\bb R_{\geq 0}$ to $M-\mathring{\convex}$ is again a diffeomorphism, and 
	\[t=d(\varbet(x,t),S)=t_2(\varbet(x,t)).\] 
	The upper bound of the second fundamental form is due to \cite[Thm.1, Prop.6]{parkkonenpaulin2012}, that is the obtained convex set $W'$ has bounded second fundamental form on its boundary.
	
	Applying Lemma \ref{prop:complement} with the set $K=\rho(M_2)$, defined in Lemma \ref{lem:proj}, and the constant $C=2(4\ell^2+1)$, there exists a bounded interval $I\in\bb R_{\geq 0}$ such that \eqref{ineq:complement1} holds for $U_2=K\times I$. Adding \eqref{ineq:complement0} and \eqref{ineq:complement1} implies the result.
\end{proof}
\begin{proof}[Proof of Proposition \ref{prop:gradient}]
	In view of \eqref{equ:enedec}, the main problem is the term $\int_{M_3^c}\|\nabla\theta\|^2f^2\dd vol$. The support of $\|\nabla\theta\|$ is contained in $\cusp-\cusp'$, which may not be compact. But with the hypothesis that the manifold has only maximal rank cusps, the region $\cusp-\cusp'$ is already relatively compact in $M$ due to Lemma \ref{lem:proj}. Because $\cusp$ does not intersect the ideal boundary and the complement of $\cusp'$ is relatively compact. Hence $\|\nabla\theta\|^2$ is a bounded function supported on $\cusp-\cusp'$.
	
	The compact set $V$ is replaced by $V\cap M_1$ and $V\cap M_2$ in Lemma \ref{lem:cusp} and Lemma \ref{lem:complement}. Using Lemma \ref{lem:cusp} and \ref{lem:complement}, we obtain $U_1$, $U_2$. Since $\cusp-\cusp',M_3,\ U_1,\ U_2$ are relatively compact in $M$, we  can find a relatively compact open set $U\subset M$ with smooth boundary, which contains the four sets. By \eqref{equ:enedec}, \eqref{ineq:cusp} and \eqref{ineq:complement}, there exists a constant $C_4$ large enough such that
	\begin{align*}
	E(f)+C_4\int_{U}f^2\dd vol\geq 
	\frac{1}{2}\int\left(\|\nabla (\varphi_1 fe^{\ell t_1})\|^2e^{-2\ell t_1}+\|\nabla(\varphi_2 fe^{\ell t_2})\|^2e^{-2\ell t_2}\right)\dd vol\\
	+\int_{M_3}(1-\varphi_1^2-\varphi_2^2)\|\nabla f\|^2\dd vol+\epsilon_V\int_V f^2\dd vol.
	\end{align*}	
We restrict our computation on $U\cap\supp \varphi_2\subset M_2$
\begin{align*}
\|\nabla(\varphi_2 fe^{\ell t_2})\|^2& \geq \frac{1}{2}(\varphi_2e^{\ell t_2})^2\|\nabla f\|^2-f^2\|\nabla(\varphi_2e^{\ell t_2})\|^2\geq e^{2\ell t_2}\left(\frac{1}{2}\varphi_2^2\|\nabla f\|^2-C_{\varphi_2}f^2\right), 
\end{align*}
where we use the estimate $\sup_{x\in U}\|\nabla\varphi_2(x)\|<\infty$, thanks to the relative compactness of $U$.
In the standard cusp region we have the same estimate. Therefore taking $C_U$ large enough, we have \eqref{ineq:gradient}.
\end{proof}

\section{Non-maximal rank}
\label{sec:app}
As stated in Section \ref{sec:positive}, we will give a proof of Proposition \ref{prop:gradient} without the assumption that $M$ has only maximal rank cusps.
\subsection{Compactification and estimate at infinite}
\label{sec:errter}
Let  $g$ be a Riemannian metric on a manifold $\mani$. We define the musical isomorphism as follows (see \cite{gallot2004riemannian} for more details). For a vector $X$ in $T_x\mani$, let $X^b$ be the unique 1-form such that $X^b(v)=g(X,v)$ for every $v\in T_x\mani$. This isomorphism gives a dual tensor field (the symmetric covariant 2-tensor fields) $g^*$ of $g$, and $(\nabla f)^b=\dd f$.  

We will consider the compactification of a Riemannian manifold $\mani=L\times \bb (0,1]$ with metric given by $g=g_1(x,y)/y^2$, where $g_1(x,y)$ is a positive definite symmetric bilinear form on $T_{(x,y)}\mani$. Now we add $y=0$ to obtain a differential manifold with boundary, called $\bar{\mani}$. By definition, we have (using local coordinate vectors $(\frac{\partial}{\partial x_i})_{1\leq i\leq n}$ and $(\dd x_i)_{1\leq i\leq n}$, $g^*$ can be written as the inverse matrix of $g$) 
\begin{equation}\label{equ:gg1}
g^*=(g_1/y^2)^*=y^2g_1^*.
\end{equation}
Suppose that $g_1^*$ can be smoothly extended to $y=0$, but $g_1^*(x,0)$ may degenerate to a positive semidefinite form, that means $g_1(x,y)$ may blow up when $y\rightarrow 0$.  
\begin{lem}\label{lem:errortheta} Assume that the Riemannian metric on $L\times(0,1]$ satisfies the above condition. Let $f$ be a smooth function on $L\times[0,1]$. For every compact subset $U$ of $L$, there exists a constant $C>0$ such that for any $(x,y)\in U\times\bb (0,1] $ we have
	$$\|\nabla f(x,y)\|^2\leq Cy^2.$$
\end{lem}

\begin{proof}
	By definition and \eqref{equ:gg1}, we have
	\begin{equation*}
	\|\nabla f\|^2=g^*(\dd f,\dd f)=y^2g_1^*(\dd f,\dd f).
	\end{equation*}  
	By the smoothness of $f$ and $g_1^*$ on the boundary and the compactness, there exists a constant $C>0$ such that
	$	\|\nabla  f\|^2\leq Cy^2.$
\end{proof}

\subsection{Manifolds with non maximal rank cusp}
We will give a proof of Proposition \ref{prop:gradient} with non-maximal rank cusps. In Section \ref{sec:positive}, the assumption that $M$ has only maximal rank cusps is only used in Lemma \ref{lem:cusp} and in the proof of Proposition \ref{prop:gradient}. The proof works similarly, and we will strengthen Lemma \ref{lem:cusp}, \ref{lem:complement} and give a control of $\int_{M_3^c}\|\nabla\theta\|^2f^2\dd vol$. 

If $\cusp$ is a standard cusp region of non-maximal rank, then the region $(\cusp-\cusp')\cap M$ is not relatively compact in $M$. Because the cusp region $\cusp$ intersects the ideal boundary $M_I=M_c-M$. To overcome this difficulty, we use the compactness in $M_c$.
The fact that the partition of unity is smooth not only on $M$ but also on $M_c$ is the key point to apply Lemma \ref{lem:errortheta}.

Recall that $M$ is a geometrically finite rank one locally symmetric manifold. We have a covering $M_c=M_1\cup M_2\cup M_3$, where $M_1$ is the cusp region and $M_2$ is a subset of the complement of the convex core.
Recall that $D=\Gamma\backslash W'$ is a neighbourhood of the convex core and $H$ is a subset of $M_1$, which is the quotient of a horoball. For $x$ in $M$, in \eqref{equ:t1t2} we have defined 
\begin{equation*}
t_1(x)=d(x,H), t_2(x)=d(x,\convex).
\end{equation*}
Recall that $\varphi_1,\varphi_2$ are two smooth functions supported on $M_1,M_2$, respectively, and $\varphi_1^2+\varphi_2^2=1$ outside of $M_3$.  Let $M_4=M_1\cap M_2-M_3$. The following lemma is the key additional ingredient for the general case.
\begin{figure}\label{fig:cuspregion}
	\begin{center}
		\begin{tikzpicture}[scale=1.7]
		\draw [domain=0:pi, samples=200] plot ({2*cos(\x r)},{2*sin(\x r)});
		\draw [domain=0:pi, samples=200] plot ({3*cos(\x r)},{3*sin(\x r)});
		\draw (-4,0) -- (4,0);
		\draw (-3,3.75) -- (-3,3.25);
		\draw (3,3.75) -- (3,3.25);
		\draw [<-] (-2.5,3.75) -- (-1,3.75);
		\draw [<-] (2.5,3.75) -- (1,3.75);
		\draw (0,3.75) node[above]{$L$};
		\draw ({sqrt(2)/2},0) -- ({3*sqrt(2)/2},{2*sqrt(2)});
		\draw (-{sqrt(2)/2},0) -- (-{3*sqrt(2)/2},{2*sqrt(2)});
		\draw ({sqrt(2)/2},0) -- ({2.2*sqrt(2)},{1.1*sqrt(2)});
		\draw (-{sqrt(2)/2},0) --  (-{2.2*sqrt(2)},{1.1*sqrt(2)});
		\draw (0,3) node[above]{$\cusp'$};
		\draw (0,2) node[above]{$\cusp$};
		\draw (0,0) node[above]{$D$};
		\draw ({sqrt(2)},0) node[below]{$C_2$};
		\draw [<-] ({sqrt(2)},5/8) -- ({sqrt(2)},-1/16);
		\draw (-4,3.5) -- (4,3.5);
		\draw (4,3.5) node[above]{$H$};
		\draw (3.5,2) node[below]{$proj_H$};
		\draw [<-] (4,3) -- (4,1);
		\draw (2.5,1.5) node[below]{$\rho$};
		\draw [->] (2.5,1.5) -- (1.8,2);
		\draw (2.5,0.5) node[below]{$M_4$};
		\draw (2.976344,0.376000) -- (3.075556,0.388533);
		\draw (2.905749,0.746070) -- (3.002608,0.770939);
		\draw (2.789329,1.104374) -- (2.882307,1.141186);
		\draw (2.628920,1.445261) -- (2.716551,1.493436);
		\draw (2.427051,1.763356) -- (2.507953,1.822134);
		\draw (2.186906,2.053641) -- (2.259803,2.122096);
		\draw (1.912272,2.311540) -- (1.976014,2.388591);
		\draw (1.607480,2.532984) -- (1.661063,2.617417);
		\draw (1.277338,2.714481) -- (1.319916,2.804964);
		\draw (0.927051,2.853170) -- (0.957953,2.948275);
		\draw (0.562144,2.946862) -- (0.580882,3.045090);
		\draw (0.188372,2.994080) -- (0.194651,3.093883);
		\draw (-0.188372,2.994080) -- (-0.194651,3.093883);
		\draw (-0.562144,2.946862) -- (-0.580882,3.045090);
		\draw (-0.927051,2.853170) -- (-0.957953,2.948275);
		\draw (-1.277338,2.714481) -- (-1.319916,2.804964);
		\draw (-1.607480,2.532984) -- (-1.661063,2.617417);
		\draw (-1.912272,2.311540) -- (-1.976014,2.388591);
		\draw (-2.186906,2.053641) -- (-2.259803,2.122096);
		\draw (-2.427051,1.763356) -- (-2.507953,1.822134);
		\draw (-2.628920,1.445261) -- (-2.716551,1.493436);
		\draw (-2.789329,1.104374) -- (-2.882307,1.141186);
		\draw (-2.905749,0.746070) -- (-3.002608,0.770939);
		\draw (-2.976344,0.376000) -- (-3.075556,0.388533);
		\draw (1.984229,0.250666) -- (2.083441,0.263200);
		\draw (1.937166,0.497380) -- (2.034025,0.522249);
		\draw (1.859553,0.736249) -- (1.952531,0.773062);
		\draw (1.752613,0.963507) -- (1.840244,1.011683);
		\draw (1.618034,1.175571) -- (1.698936,1.234349);
		\draw (1.457937,1.369094) -- (1.530834,1.437549);
		\draw (1.274848,1.541026) -- (1.338590,1.618078);
		\draw (1.071654,1.688656) -- (1.125236,1.773089);
		\draw (0.851559,1.809654) -- (0.894137,1.900137);
		\draw (0.618034,1.902113) -- (0.648936,1.997219);
		\draw (0.374763,1.964575) -- (0.393501,2.062803);
		\draw (0.125581,1.996053) -- (0.131860,2.095856);
		\draw (-0.125581,1.996053) -- (-0.131860,2.095856);
		\draw (-0.374763,1.964575) -- (-0.393501,2.062803);
		\draw (-0.618034,1.902113) -- (-0.648936,1.997219);
		\draw (-0.851559,1.809654) -- (-0.894137,1.900137);
		\draw (-1.071654,1.688656) -- (-1.125236,1.773089);
		\draw (-1.274848,1.541026) -- (-1.338590,1.618078);
		\draw (-1.457937,1.369094) -- (-1.530834,1.437549);
		\draw (-1.618034,1.175571) -- (-1.698936,1.234349);
		\draw (-1.752613,0.963507) -- (-1.840244,1.011683);
		\draw (-1.859553,0.736249) -- (-1.952531,0.773062);
		\draw (-1.937166,0.497380) -- (-2.034025,0.522249);
		\draw (-1.984229,0.250666) -- (-2.083441,0.263200);
		\end{tikzpicture}
	\end{center}
	\caption{Projection in standard cusp region}
\end{figure}
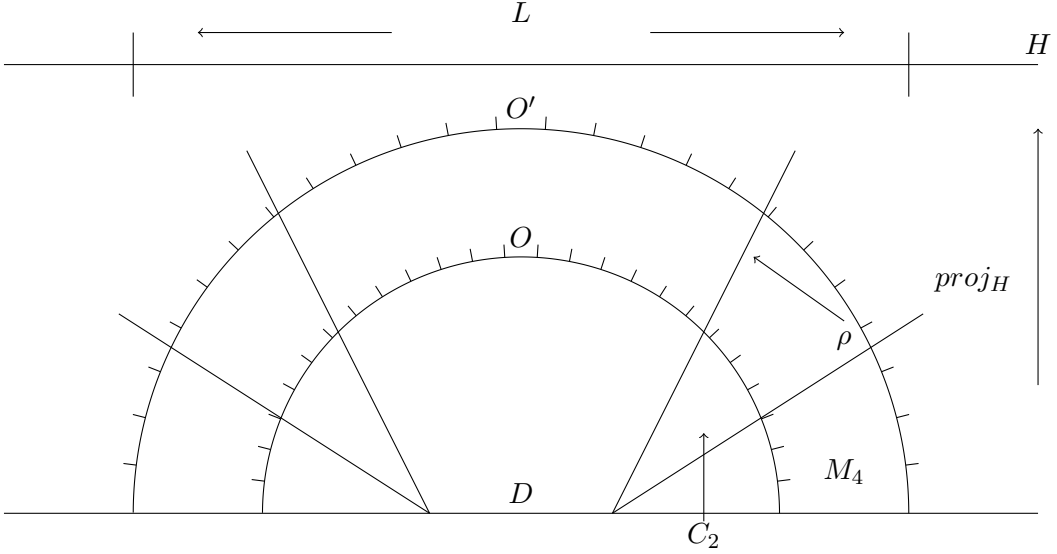
\begin{lem}\label{lem:nabf}
	For every smooth function $ f$ on $M_4$, there exists $C_ f>0$, such that for all $x\in M_4\cap M$
	\begin{equation}\label{ineq:theta}
	\|\nabla f(x)\|^2\leq C_{ f}\left(e^{-t_1(x)}\varphi_1(x) ^2+e^{-t_2(x)}\varphi_2(x) ^2\right).
	\end{equation}
\end{lem}
\begin{proof}
	By \eqref{equ:duamet}, letting $y=e^{-t}$, the metric in the half space model is given by $g=g_y\oplus \frac{\dd y^2}{y^2}$ where $g_y=\begin{pmatrix}
	\frac{1}{y^2}Id_m & 0 \\ 0 & \frac{1}{y^4}Id_q
	\end{pmatrix}$ under some distributions $W^1$ and $W^2$ on $N$, which satisfies the condition in Lemma \ref{lem:errortheta}. 
	Due to $M_4\subset M_1\subset M_c(\Pi)=(\Pi\backslash N)\times\Rplus$, where $\Pi$ is a discrete parabolic subgroup of $\Gamma$ which preserves the metric $g_y$ on $N$ for every $y$ in $\Rplus$, the quotient metric on $M_4$ also satisfies the condition in Lemma \ref{lem:errortheta}. Moreover, by $M_4\subset M_1\cap M_2\subset O-O'$, we have that $M_4$ is contained in $L\times[0,1]$, where $L=proj_H(O-O')$ is relatively compact by Lemma \ref{lem:proj}.
	Therefore, using Lemma \ref{lem:errortheta} with $U=L$ there exists $C_f'>0$ such that
	\begin{align}\label{equ:nabla f}
	\|\nabla f(\cdot,y)\|^2\leq C_f'y^2.
	\end{align}
	Due to $M_4=M_1\cap M_2-M_3$ and the definition of $\varphi_1,\varphi_2$, we have $1=\varphi_1^2(x)+\varphi_2^2(x)$ for $x$ in $M_4$, and the right-hand side of the above inequality equals $C_f'y^2(|\varphi_1|^2+|\varphi_2|^2)$.
	
	We restrict our attention to $x$ in $M_4\cap M$. By the argument of the proof of Lemma \ref{lem:cusp}, we have 
	\begin{equation}\label{equ:y t1}
		y=e^{-t_1(x)}.
	\end{equation}
	By definition of $K,L$, the nearest points of $x$ in $H$ and $ \convex$ are contained in $K$ and $L$. Therefore by compactness of $K, L$ (Lemma \ref{lem:proj}), we have
	\begin{equation}\label{ineq:t1t2}
	|t_1(x)-t_2(x)|=|d(x,H)-d(x, \convex)|\leq \sup\{d(x_1,x_2)|x_1\in K,x_2\in L \}=C_{K,L}.
	\end{equation}
	Hence $e^{-t_1(x)}\leq e^{C_{K,L}}e^{-t_2(x)}$ for $x$ in $M_4\cap M$. Therefore on $M_4\cap M$, by \eqref{equ:nabla f},\eqref{equ:y t1} and \eqref{ineq:t1t2}
	\[\|\nabla f(x)\|^2\leq C_f'e^{-2t_1(x)}(\varphi_1^2(x)+\varphi_2^2(x))\leq C_f'(e^{-2t_1(x)}\varphi_1^2(x)+e^{2C_{K,L}}e^{-2t_2(x)}\varphi_2^2(x)).\]
	The proof is complete due $t_1,t_2\geq 0$.
\end{proof}
We state our strengthened version of Lemma \ref{lem:cusp} and \ref{lem:complement}.
\begin{replem}{lem:cusp} 
	In the standard cusp region, for every $C>0$ there exist a compact set $U_1$ in $M_1\cap M$ and a constant $C_1>0$ such that the following holds. For any compact set $V$ in $M_1\cap M$ there exists $\epsilon_V>0$ such that for all $f\in C_c^{\infty}(M_1\cap M)$ we have
	\begin{equation}\label{ineq:cusp*} 
	\begin{split}
	E(f)+C_1\int_{U_1}f^2\dd vol\geq \frac{1}{2}\int\|\nabla (e^{\ell t_1}f)\|^2e^{-2\ell t_1}\dd vol+\epsilon_V\int_V f^2\dd vol+C\int_{M_4}e^{-t_1}f^2\dd vol.
	\end{split}
	\end{equation}
\end{replem}
We need a variant of Poincar\'e's inequality.
\begin{lem}
	For relatively compact sets $V_1, V_2$ in a Riemannian manifold $\mani$, where $V_1$ is connected open with smooth boundary and $V_2$ is a subset of $V_1$ with nonempty interior, there exists a positive constant $\epsilon$, such that for all $g\in C_c^\infty(\mani)$ we have
	\begin{equation}\label{ineq:poincare}
	\int_{V_1}\|\nabla g\|^2\dd vol+\int_{V_2}g^2\dd vol\geq \epsilon\int_{V_1}g^2\dd vol.
	\end{equation}
\end{lem}
\begin{proof}
	Suppose that the result does not hold. Recall the notion of Sobolev spaces from Section \ref{sec:finite}. Then for every $n\in\bb N$ there exists a smooth function $g_n$ such that $\|g_n\|_{H^1(V_1)}=1$ and
	\begin{equation}\label{equ:gn}
		\int_{V_1}\|\nabla g_n\|^2\dd vol+\int_{V_2}g_n^2\dd vol\leq \frac{1}{n}\int_{V_1}g_n^2\dd vol\leq \frac{1}{n}.
	\end{equation}
	By Banach-Alaoglu's theorem, there exists a subsequence $\{g_{n_j}\}$ which converges weakly to a function $g$ in $H^1(V_1)$. By the Rellich theorem \cite[Chapter 4, Proposition 4.4]{taylor1}, the injection $H^1(V_1)\rightarrow L^2(V_1)$ is compact. Then the subsequence $\{g_{n_j}\}_{j\in\bb N}$ converges strongly to the same function $g$ in $L^2(V_1)$. 
	Now by weakly convergence
	\begin{equation}\label{equ:gH}
	\|g\|_{H^1(V_1)}\leq \liminf_{j\rightarrow \infty}\|g_n\|_{H^1(V_1)}=1.
	\end{equation}
	But by \eqref{equ:gn}, we have $\|g_n\|^2_{L^2(V_1)}\geq 1-\frac{1}{n}$, and the strong convergence implies $\|g\|^2_{L^2(V_1)}=1$. Due to \eqref{equ:gH}, we obtain
	\[\int_{V_1}\|\nabla g\|^2\dd vol =0. \]
	Hence the limit function $g$ is constant. 
	
	Again by \eqref{equ:gn}, we have
	$\int_{V_2}g_n^2\dd vol\leq \frac{1}{n}$.
	Taking limit, we have $\int_{V_2}g^2\dd vol\leq 0$, 
	which implies that $g$ is constant zero on $V_2$. Since $g$ is constant, we see that $g=0$, which contradicts the fact that $\|g\|_{L^2(V_1)}=1$. 
\end{proof}
\begin{proof}[Proof of Lemma* \ref{lem:cusp}]
By Lemma \ref{lem:proj}, we have $M_4\subset L\times\bb R_{\geq 0}$. By the same argument as in the proof of Lemma \ref{lem:cusp}, using Proposition \ref{prop:cusp}, we have the above inequality with $\int_Vf^2$ replaced by $\int_{L\times \bb R_{\geq 0}}e^{-t_1}f^2$. 

Applying \eqref{ineq:poincare} with $g=e^{\ell t_1}f$, $V_2=L\times[1,2]$ and $V_1$ a relatively compact connected open set with smooth boundary containing $V_2\cup V$ implies that
	\begin{align*}
	\int_{V_1}\|\nabla (e^{\ell t_1}f)\|^2e^{-2\ell t_1}\dd vol+&\int_{V_2}e^{-t_1}f^2\dd vol \geq \epsilon_1\left(\int_{V_1}\|\nabla g\|^2\dd vol+\int_{V_2}g^2\dd vol\right)\\
	&\geq \epsilon_2\int_{V_1} g^2\dd vol=\epsilon_2\int_{V_1} f^2e^{2\ell t_1}\dd vol\geq \epsilon_3\int_V f^2\dd vol.
	\end{align*}
The proof is complete.
\end{proof}

\begin{replem}{lem:complement}
	In the complement of the convex core, for every $C>0$ there exist a compact set $U_2$ in $M_2\cap M$ and a constant $C_2>0$ such that the following holds. For any compact set $V$ in $M_2\cap M$ there exists $\epsilon_V>0$ such that for all $f\in C_c^{\infty}(M_2\cap M)$ we have
	\begin{align}
	E(f)+C_2\int_{U_2}f^2 \dd vol\geq &\frac{1}{2}\int\|\nabla (e^{\ell t_2}f)\|^2e^{-2\ell t_2}\dd vol+\epsilon_V\int_Vf^2 \dd vol+C\int_{M_4}e^{-t_2}f^2 \dd vol. \label{ineq:complement*}
	\end{align}
\end{replem}
\begin{proof}
By Lemma \ref{lem:proj}, we have $M_4\subset M_2\subset \varbet(K\times \bb R_{\geq 0})$. Then the proof is exactly the same as the proof of Lemma \ref{lem:complement}, using Lemma \ref{prop:complement} with the constant $2(4\ell^2+1)+C$.
\end{proof}
\begin{proof}[Proof of Proposition \ref{prop:gradient}]
		Applying Lemma \ref{lem:nabf} with $\theta=\frac{\pi}{2}\bar{\varphi_3}$, we get a constant $C_\theta$ such that \eqref{ineq:theta} holds for $\theta$. Using Lemma$^*$ \ref{lem:cusp} and \ref{lem:complement} with $C=C_\theta$, we obtain $U_1$, $U_2$.
	Then follow the same argument as in the proof of the special case of Proposition \ref{prop:gradient}. The proof is complete. 
\end{proof}

\noindent Jialun LI\\
Université de Bordeaux\\
jialun.li@math.u-bordeaux.fr
\end{document}